\documentclass[preprint,11pt]{elsarticle}

\usepackage{amsfonts, amsmath, amscd}
\usepackage[psamsfonts]{amssymb}

\usepackage{amssymb,textcomp,lmodern}

\usepackage{pb-diagram}

\usepackage[all,cmtip]{xy}

\usepackage[usenames]{color}

\newtheorem{theorem}{Theorem}[section]
\newtheorem{lemma}[theorem]{Lemma}
\newtheorem{corollary}[theorem]{Corollary}

\newtheorem{remark}[theorem]{Remark}
\newtheorem{proposition}[theorem]{Proposition}
\newtheorem{definition}[theorem]{Definition}
\newtheorem{example}[theorem]{Example}

\newtheorem{claim}[theorem]{Claim}
\newtheorem{notation}[theorem]{Notation}
\newtheorem{problem}[theorem]{Problem}

\newproof{proof}{Proof}
\newproof{claimproof}{\it Proof}

\numberwithin{equation}{section}

\newcommand{\e}{\varepsilon}
\newcommand{\w}{\omega}

\newcommand{\NN}{\mathbb{N}}

\newcommand{\IR}{\mathbb{R}}

\newcommand{\II}{\mathbb{I}}

\newcommand{\yyy}{\mathbf{y}}

\newcommand{\aaa}{\mathbf{a}}

\newcommand{\nn}{\mathfrak{n}}
\newcommand{\Pp}{\mathfrak{P}}

\newcommand{\FF}{\mathcal{F}}
\newcommand{\F}{\mathcal{F}}
\newcommand{\V}{\mathcal{V}}
\newcommand{\U}{\mathcal{U}}
\newcommand{\W}{\mathcal{W}}

\newcommand{\SC}{\mathcal{S}}
\newcommand{\KK}{\mathcal{K}}
\newcommand{\Nn}{\mathcal{N}}
\newcommand{\AAA}{\mathcal A}
\newcommand{\I}{\mathcal{I}}

\newcommand{\A}{\mathcal{A}}
\newcommand{\C}{\mathcal{C}}

\newcommand{\cl}{\mathrm{cl}}

\newcommand{\Lin}{\mathrm{Lin}}

\newcommand{\CC}{C_k}

\newcommand{\SM}{{\setminus}}

\input xy
\xyoption{all}

\begin{document}

\begin{frontmatter}

\title{Topological properties of some function spaces}


\author{Saak Gabriyelyan}
\ead{saak@math.bgu.ac.il}
\address{Department of Mathematics, Ben-Gurion University of the Negev, Beer-Sheva, P.O. 653, Israel}

\author{Alexander V. Osipov}
\ead{OAB@list.ru}
\address{Krasovskii Institute of Mathematics and Mechanics, Ural Federal  University, \\ Ural State University of Economics, Yekaterinburg, Russia}

\begin{abstract}
Let $Y$ be a metrizable space containing at least two points, and let $X$ be a $Y_\I$-Tychonoff space for some ideal $\I$ of compact sets of $X$. Denote by $C_\I(X,Y)$ the space of continuous functions from $X$ to $Y$ endowed with the $\I$-open topology. We prove that $C_\I(X,Y)$ is Fr\'{e}chet--Urysohn iff $X$ has the property $\gamma_\I$. We characterize  zero-dimensional Tychonoff spaces $X$ for which the space  $C_\I(X,\mathbf{2})$ is sequential.
Extending the classical theorems of Gerlits, Nagy and Pytkeev we show that if $Y$ is not compact, then  $C_p(X,Y)$ is Fr\'{e}chet--Urysohn iff it is sequential iff it is a $k$-space iff $X$ has the property $\gamma$.
An analogous result is obtained for the space of bounded continuous functions taking values in a metrizable locally convex space.
Denote by $B_1(X,Y)$ and $B(X,Y)$ the space of Baire one functions and the space of all Baire functions from $X$ to $Y$, respectively. If $H$ is a subspace of  $B(X,Y)$ containing $B_1(X,Y)$, then $H$ is metrizable iff it is a $\sigma$-space iff it has countable $cs^\ast$-character iff $X$ is countable. If additionally $Y$ is not compact, then $H$ is Fr\'{e}chet--Urysohn iff it is sequential iff it is a $k$-space iff it has countable tightness iff $X_{\aleph_0}$ has the property $\gamma$, where $X_{\aleph_0}$ is the space $X$ with the Baire topology. We show that if $X$ is a Polish space, then the space $B_1(X,\IR)$ is normal iff $X$ is countable.
\end{abstract}


\end{frontmatter}



\section{Introduction}


For Tychonoff spaces $X$ and $Y$, we denote by $C_p(X,Y)$ and $\CC(X,Y)$ the family $C(X,Y)$ of all continuous functions from $X$ to $Y$ endowed with the topology of pointwise convergence or the compact-open topology, respectively. If $Y=\IR$, we shall write $C_p(X)$ and $\CC(X)$.

The study of topological properties of spaces of continuous functions is quite an active area of research attracting specialists both from General Topology and Functional Analysis and has a long history. Moreover, the study of topological properties of the function spaces $C_p(X)$ and $\CC(X)$ is one of the main topics in General Topology. For numerous results and historical remarks we refer the reader to  the classical texts \cite{Arhangel,mcoy} or to the recent monograph \cite{Tkachuk-Book-2010} and references therein. Let us recall some of the most famous results (all relevant definitions are given below or can be found for example in the classical book of Engelking \cite{Eng}). 
We start from the following two remarkable theorems proved by Pol in his seminal paper \cite{Pol-1974}, where $\II=[0,1]$.
\begin{theorem}[Pol] \label{t:Pol-p-normal}
If $X$ is a metric space, then the following assertions are equivalent: (a) $C_p(X,\II)$ is normal, (b) $C_p(X,\II)$ is Lindel\"{o}f, (c) the set of all non-isolated points of $X$ is separable.
\end{theorem}
\begin{theorem}[Pol] \label{t:Pol-k-space}
\begin{enumerate}
\item[{\rm (i)}] If $X$ is a first countable paracompact space, the following assertions are equivalent: (a) the space $\CC(X,\II)$ is a $k$-space,  (b) for any compact metrizable space $K$, the space $\CC(X,K)$ is paracompact and \v{C}ech-complete, (c) $X=L\cup D$ is the topological sum of a locally compact Lindel\"{o}f space $L$ and a discrete space $D$.
\item[{\rm (ii)}] If $X$ is a metric space, the following assertions are equivalent: (a) $\CC(X,\II)$ is normal, (b) $\CC(X,\II)$ is Lindel\"{o}f, (c) for any compact metrizable space $K$, the space $\CC(X,K)$ is Lindel\"{o}f, (d) the set of all non-isolated points of $X$ is separable.
\end{enumerate}
\end{theorem}

The following fundamental result was proved independently by Gerlits and Nagy \cite{GN,Gerlits} and Pytkeev \cite{Pytkeev-Cp}.
\begin{theorem}[Gerlits--Nagy--Pytkeev] \label{t:GNP}
For a Tychonoff space $X$ the following assertions are equivalent:
\begin{enumerate}
\item[{\rm (i)}] $C_p(X)$ is a Fr\'{e}chet--Urysohn space;
\item[{\rm (ii)}]  $C_p(X)$  is a sequential space;
\item[{\rm (iii)}]  $C_p(X)$  is a $k$-space;
\item[{\rm (iv)}]  $X$ has the property $\gamma$.
\end{enumerate}
\end{theorem}

For a Tychonoff space $X$ we denote by $C^b(X)$ the subspace of $C(X)$ containing all bounded bounded functions.
In \cite{Pytkeev-Ck}, Pytkeev proved the following important result which we shall use below (for the definition of an ideal $\I$ of compact sets in $X$ and the $\I$-open topology see Section \ref{sec:covering}). Although this theorem is proved in \cite{Pytkeev-Ck} only for $C_\I(X)$, the proof for the space $C^b_\I(X)$ is exactly the same because all functions used in the proof can be chosen bounded.
\begin{theorem}[Pytkeev] \label{t:Pytkeev-Ck}
Let $X$ be a Tychonoff space, $\I$ be an ideal of compact subsets of $X$ and let $E=C_\I(X)$ or $E=C^b_\I(X)$. Then the following assertions are equivalent:
\begin{enumerate}
\item[{\rm (i)}] $E$ is a Fr\'{e}chet--Urysohn space;
\item[{\rm (ii)}] $E$ is a sequential space;
\item[{\rm (iii)}] $E$ is a $k$-space.
\end{enumerate}
\end{theorem}

Denote by $\mathbf{2}$ a discrete two element space. For  a zero-dimensional Polish space $X$, a characterization of sequentiality of  $\CC(X,\mathbf{2})$  was obtained in  \cite{GTZ}. The following more general theorem was proved by the first author.

\begin{theorem}[\cite{Gabr-C2}] \label{t:Ascoli-C(X,2)-A}
Let $X$ be  a zero-dimensional metric space $X$. Then:
\begin{enumerate}
\item[{\rm (i)}] $\CC(X,\mathbf{2})$ is a $k$-space if and only if either $X=L\cup D$ is a topological sum of a separable metrizable locally compact space $L$ and a discrete space $D$ or $X$ is not locally compact but the set $X'$  of non-isolated points of $X$ is compact;
\item[{\rm (ii)}] $\CC(X,\mathbf{2})$ is a sequential space  if and only if $X$ is a Polish space and either $X$ is locally compact or $X$ is  not locally compact but the set $X'$ of non-isolated points of $X$ is compact;
\item[{\rm (iii)}] $\CC(X,\mathbf{2})$ is a Fr\'{e}chet--Urysohn space  if and only if $\CC(X,\mathbf{2})$ is a Polish space if and only if $X$ is a Polish  locally compact space.
\end{enumerate}
\end{theorem}
Let us also mention the following interesting result of Pol and Smentek \cite{PolSmen}: {\em If $X$ is a zero-dimensional realcompact $k$-space, then the group $\CC(X,\mathbf{2})$  is reflexive.} For other results concerning topological properties of $\CC(X,\mathbf{2})$ see \cite{Banakh-Survey,Gabr-C2}.

Let $X$ be a Tychonoff space and let $E$ be a locally convex space. In the theory of locally convex spaces there is a long tradition to investigate locally convex properties of the space $C(X,E)$ endowed with the pointwise topology or the compact-open topology by means of topological properties of $X$ and locally convex properties of $E$. The most important case is the case when $E$ is a Banach space (so if additionally $X$ is compact, we obtain the widely studied class of Banach spaces). For numerous results obtained in the eighties of the last century and historical remarks we refer the reader to the well known lecture notes of Schmets \cite{Schmets}.

The aforementioned results motivate the following general problem.
\begin{problem} \label{prob:gen-top-C(X)}
Let $Y$ be a metric space containing at least two points, $\I$ be an ideal of compact sets in a topological space $X$, and let $\mathcal{P}$ be a topological property. Characterize (in terms of $X$ and $Y$) the spaces $C_\I(X,Y)$ with $\mathcal{P}$.
\end{problem}
In this paper we concentrate mainly on the properties of being a Fr\'{e}chet--Urysohn, sequential, normal space or a $k$-space.

If $X$ is connected and $Y$ is discrete, then the function space $C(X,Y)$ contains only constant functions. So to avoid such  unpleasant cases and to have the space $C(X,Y)$ sufficiently rich, we have to consider several separation axioms by analogy with the classical notion of  Tychonoff spaces (notice that exactly by this reason the space $X$ in \cite{GTZ,PolSmen} and Theorem \ref{t:Ascoli-C(X,2)-A} is assumed to be zero-dimensional). For the topology of pointwise convergence such axioms were considered in \cite{BG-Baire}.
In Section \ref{sec:covering}, we define and study necessary separation axioms and covering properties depending on an ideal of compact subsets of $X$ which are essential for our main results.

In Section \ref{sec:FU-sequentiality}, in terms of covers of the space $X$  we obtain the following results: (a) a characterization of spaces $X$ for which the space $C_\I(X,Y)$ is Fr\'{e}chet--Urysohn (Theorem \ref{t:Cp(X,Y)-FU}), and (b) a characterization of zero-dimensional $T_1$-spaces $X$ for which the space $C_\I(X,\mathbf{2})$ is sequential (Theorem \ref{t:CI-2-sequential}).
In Section \ref{sec:k-normal}, for a non-compact metric space $Y$ we show that the conclusion of the Gerlits--Nagy--Pytkeev Theorem \ref{t:GNP} still remains true:
$C_p(X,Y)$ is a Fr\'{e}chet--Urysohn space if and only if it is a sequential space if and only if it is a $k$-space if and only if $X$ has the property $\gamma$ (see  Theorem \ref{t:H-k-space}).  Applying Pytkeev's Theorem \ref{t:Pytkeev-Ck},  we show  in  Theorem \ref{t:Pytkeev-Ck-Cb} that its conclusion is true also for the spaces $C_\I^b(X,Y)$ and $C_\I^{rc}(X,Y)$, where $Y$ is a metrizable locally convex space and $C_\I^{rc}(X,Y)$ denotes the space of all continuous functions $f$ whose image $f(X)$ is relatively compact in $Y$.

Although the classes of continuous functions are the most important, there are other classes of (noncontinuous) functions which are of significant importance and widely studied in General Topology and Analysis; for example, the classes of Baire type functions introduced and studied by Baire \cite{Baire}.

Let $X$ and $Y$ be topological spaces. For $\alpha=0$, we put $B_0(X,Y):=C_p(X,Y)$. For every nonzero countable ordinal $\alpha$, let $B_\alpha(X,Y)$ be the family of all functions $f:X\to Y$ that are pointwise limits of sequences $\{f_n\}_{n\in\w}\subseteq \bigcup_{\beta<\alpha}B_\beta(X,Y)$ and $B(X,Y):=\bigcup_{\alpha<\w_1}B_\alpha(X,Y)$. All the spaces $B_\alpha(X,Y)$ and $B(X,Y)$ are endowed with the topology of pointwise convergence, inherited from the Tychonoff product $Y^X$. If $Y=\IR$, set $B_\alpha(X):=B_\alpha(X,\IR)$ and $B(X):=B(X,\IR)$.

The most important case is the case when $X$ is a Polish space. This is explained not only by the classical results due to Ren\'e Baire, Henri Lebesgue and others, but also  by some deep results in the Banach space theory. If $X$ is a Polish space, the compact subsets of $B_1(X)$ (called {\em Rosenthal compact}) have been studied intensively also by Godefroy \cite{Godefroy}, Todor\v{c}evi\'{c} \cite{Todorcevic} and many others. For a Banach space $Y$, the compact subsets of $B_1(X,Y)$ are called {\em Rosenthal--Banach compacts}; they  were introduced and studied in \cite{Mer-Sta}.
The study of topological properties of spaces of Baire functions is also motivated by the following fundamental theorem which is proved by Bourgain, Fremlin and Talagrand \cite{BFT} and is a strengthening of a result of Rosenthal \cite{Rosenthal-compact}:

\begin{theorem}[Bourgain--Fremlin--Talagrand] \label{t:angelic-BFT}
If $X$ is a Polish space, then $B_1(X)$ is angelic.
\end{theorem}

In a more general setting when $X$ is a Tychonoff space, some of topological properties of the spaces $B_\alpha(X), \alpha\in(0,\w_1],$ were studied by Pestryakov in his thesis \cite{Pestryakov}. Among the others, he proved the following results.
\begin{theorem}[Pestryakov] \label{t:Pest-k-space}
Let $X$ be a Tychonoff space and $0<\alpha\leq \omega_1$.
\begin{enumerate}
\item[{\rm (A)}] The following assertions are equivalent:
\begin{enumerate}
\item[{\rm (i)}] $B_{\alpha}(X)$ is a Fr\'{e}chet--Urysohn space;
\item[{\rm (ii)}] $B_{\alpha}(X)$ is a sequential space;
\item[{\rm (iii)}] $B_{\alpha}(X)$ is a  $k$-space;
\item[{\rm (iv)}] $B_{\alpha}(X)$ has countable tightness;
\item[{\rm (v)}] $X_{\aleph_0}$ is a Lindel\"{o}f space;
\item[{\rm (vi)}] $X_{\aleph_0}$ satisfies the property $\gamma$.
\end{enumerate}
\item[{\rm (B)}] $t\big(B_{\alpha}(X)\big)=\sup\{l(X^n_{\aleph_0}): n\in \omega \}$.
\end{enumerate}
\end{theorem}

The next theorem was proved recently by the first author.
\begin{theorem}[\cite{Gabr-B1}] \label{t:Baire-class}
Let $G$ be a non-precompact abelian metrizable group, $X$ a $G$-Tychonoff first countable space and let $H$ be a subgroup of $G^X$ containing $B_1(X,G)$. Then the following assertions are equivalent:
\begin{enumerate}
\item[{\rm (i)}] $X$ is countable;
\item[{\rm (ii)}] $H$ is a metrizable space and $H=G^X$;
\item[{\rm (iii)}] $H$ has countable tightness;
\item[{\rm (iv)}] $H$ has countable $cs^\ast$-character;
\item[{\rm (v)}] $H$ is a $\sigma$-space;
\item[{\rm (vi)}] $H$ is a $k$-space.
\end{enumerate}
If in addition $B_{2}(X,G)\subseteq H$, then {\em (i)-(vi)} are equivalent to
\begin{enumerate}
\item[{\rm (vii)}] $H$ is a normal space.
\end{enumerate}
\end{theorem}

In Section \ref{sec:Baire-topology} we essentially extend Theorems \ref{t:Pest-k-space} and \ref{t:Baire-class}, see Corollaries \ref{c:B1-k-space} and \ref{c:B1-metric-space}, respectively (see also Theorem \ref{t:Baire-bounded}).

To get our main results of Sections \ref{sec:FU-sequentiality}-\ref{sec:Baire-topology}, we actively use the next two ideas. The first one was successfully applied in  \cite{Gabr-B1} to prove Theorem \ref{t:Baire-class}, and it is that instead of the whole function space $C(X,Y)$ we consider only some of its sufficiently rich and saturated subspaces, see Definition \ref{def:relatively-Y-Tych}. For example, if $Y$ is a locally convex space, then the spaces $C_\I^b(X,Y)$ and $C_\I^{rc}(X,Y)$ satisfy Definition \ref{def:relatively-Y-Tych} (Proposition \ref{p:Y-LCS=>good-I}).  The second general idea used in Section \ref{sec:Baire-topology} is standard in the theory of spaces of Baire function (see for example \cite{Jayne,leri}), and it is that we reduce the study of the spaces of Baire  functions to the study of  subspaces of the space $C_p(X_{\aleph_0},Y)$, where $X_{\aleph_0}$ is the set $X$ endowed with the Baire topology (for details see Section \ref{sec:covering}). It turns out (Proposition \ref{p:G-delta-D-Tychonoff}) that already the space $B_{1}(X,Y)$ considered as a subspace of $C_p(X_{\aleph_0},Y)$ is sufficiently rich in the sense of Definition \ref{def:relatively-Y-Tych} and, therefore, we can apply our main results of Sections \ref{sec:FU-sequentiality} and  \ref{sec:k-normal}    to get essential generalizations of Theorems \ref{t:Pest-k-space} and \ref{t:Baire-class}.

By Theorem \ref{t:Baire-class}, if $\alpha>1$, then the space $B_{\alpha}(X,G)$ is normal if and only if $X$ is countable. However, the question of whether the same holds true also for the space $B_{1}(X,G)$ of functions of the first Baire class remains open. In the last Section \ref{sec:normal} we obtain a partial answer to this question in the most important case of Polish spaces $X$ proving the following result (Corollary \ref{c:B1-Polish-normal}): {\em If $X$ is a Polish space, then $B_1(X)$ is normal if and only if $X$ is countable.}


\section{Some covering properties and separation axioms} \label{sec:covering}


In this section we consider several covering properties and separation axioms which are essentially used in the paper.

\subsection{Some covering properties}
We start from some necessary definitions and notations.
Let $X$ be an arbitrary set and let $\gamma$ be a family of subsets of $X$. Set
\[
\underline{\lim}\, \gamma := \big\{ x\in X: \mbox{ the set } \{ U\in\gamma: x\not\in U\} \mbox{ is finite} \big\}.
\]
If $\gamma =\{ U_n\}_{n\in\w}$, then $\underline{\lim}\,\gamma =\underline{\lim}\, U_n =\bigcup_{n\in\w}\bigcap_{i\geq n} U_i$, and we also write $U_n \to \underline{\lim}\, \gamma$; if additionally $\underline{\lim}\, \gamma=X$
 the sequence $\gamma$ is called a {\em $p$-sequence}.

If $\U$ and $\U'$ are families of subsets of $X$, say that $\U\leq \U'$ if, for any $U\in\U$, there is $U'\in\U'$ such that $U\subseteq U'$. For families  $\U_1,\dots,\U_n$ of subsets of $X$, set
\[
\U_1 \wedge \cdots \wedge \U_n := \{ U_1\cap\cdots\cap U_n: U_i\in\U_i \mbox{ for all } i=1,\dots,n\}.
\]
It is clear that $\U_1 \wedge \cdots \wedge \U_n \leq \U_i$ for every $i=1,\dots,n$.

A sequence $\{ \U_n\}_{n\in\w}$ of subsets of $X$ is called {\em decreasing} if $\U_{n+1} \leq \U_n$ for each $n\in\w$.

Let $X$ be a topological space. A family  $\gamma$ of subsets of $X$ is called an {\em $\w$-cover of $X$} if
for any finite $A\subseteq X$, there is $U\in\gamma$ such that $A\subseteq U$.
Recall (see \cite{Arhangel}) that the space $X$ is said to have
\begin{enumerate}
\item[$\bullet$]  the {\em property $\gamma$} if  every open $\w$-cover contains a $p$-sequence;
\item[$\bullet$] the  {\em property $\varphi$} if  for any open $\w$-cover $\eta =\{ \eta_n: n\in\w\}$ of $X$ with $\eta_n\leq \eta_{n+1}$ there exists a $p$-sequence $\xi=\{ X_n: n\in\w\}$ such that $X_n$ is $\w$-covered by $\eta_n$ for every $n\in\w$;
\item[$\bullet$] the {\em property $\varepsilon$} if  any open $\w$-cover $\eta $ of $X$ has a countable $\w$-subcover.
\end{enumerate}

The following proposition is Theorem 1 in \cite{Gerlits} and Theorem II.3.2 of \cite{Arhangel}.

\begin{proposition} \label{p:property-gamma}
For a topological space $X$ the following assertions are equivalent:
\begin{enumerate}
\item[{\rm (i)}] $X$ has the property $\gamma$;
\item[{\rm (ii)}] $X$ has  the properties $\varphi$ and $\e$;
\item[{\rm (iii)}] for any sequence $\{ \gamma_n\}_{n\in\w}$ of open $\w$-covers  of $X$, one can choose $U_n\in\gamma_n$ for each $n\in\w$, such that $\underline{\lim}\, U_n=X$.
\end{enumerate}
\end{proposition}

Below we generalize the above-mentioned notions to any ideal of compact sets in $X$.
Recall that a family $\I$ of compact subsets of a topological space $X$ is called an {\em ideal of compact sets} if $\bigcup\I=X$ and for any sets $A,B\in\I$ and any compact subset $K\subseteq X$ we get $A\cup B\in\I$ and $A\cap K\in\I$, i.e. if $\I$ covers $X$ and is closed under taking finite unions and closed subspaces. The most important cases are the ideal $\I=\F(X)=[X]^{<\w}$ of all finite subsets of $X$ and the ideal $\I=\KK(X)$ of all compact subsets of $X$.

Let $X$ be a topological space and let $\I$ be an ideal of compact sets of $X$.  A family  $\gamma$ of subsets of $X$ is called an {\em $\I$-cover of $X$} if for any $A\in \I$, there is $U\in\gamma$ such that $A\subseteq U$,   see \cite[\S~4.4]{mcoy}.
An  {\em $\I$-sequence} in $X$ is any sequence $\{ C_n\}_{n\in\w}$ of subsets of $X$ with the property that if $A\in\I$, then there exists an $m\in\w$ such that $A\subseteq C_n$ for all $n\geq m$.
So, if $\I=\FF(X)$ or $\I=\KK(X)$  and following \cite{mcoy}, we shall say that an $\I$-cover is a {\em $p$-cover} (=an $\w$-cover) or a {\em $k$-cover} of $X$ and an $\I$-sequence is a $p$-sequence or a $k$-sequence, respectively.
Analogously to the properties $\gamma$, $\varphi$ and $\e$, we say that the space $X$ has
\begin{enumerate}
\item[$\bullet$]  the {\em property $\gamma_\I$} if  every open $\I$-cover contains an $\I$-sequence;
\item[$\bullet$] the  {\em property $\varphi_\I$} if  for any open $\I$-cover $\eta =\{ \eta_n: n\in\w\}$ of $X$ with $\eta_n\leq \eta_{n+1}$ there exists an $\I$-sequence $\xi=\{ X_n: n\in\w\}$ such that $X_n$ is $\I$-covered by $\eta_n$ for every $n\in\w$;
\item[$\bullet$] the {\em property $\varepsilon_\I$} if  any open $\I$-cover $\eta $ of $X$ has a countable $\I$-subcover.
\end{enumerate}

\begin{remark} \label{r:varphi-I} {\em
Let $\eta =\{ \eta_n: n\in\w\}$ be an open $\I$-cover of a Tychonoff space $X$ such that $\eta_n\leq \eta_{n+1}$ for $n\in\w$, and let  $\xi=\{ X_n: n\in\w\}$ be an $\I$-sequence in $X$. For the property $\varphi_\I$, it is not necessary that $X_n$ is $\I$-covered exactly by $\eta_n$, it suffices to assume that $X_n$ is $\I$-covered by some $\eta_{k_n}$. Indeed, we can assume that $0\leq k_0 <k_1<\cdots$. If $0\leq k<k_1$, set $X'_k:=\emptyset$. For every $n\geq 1$ and $k_n\leq k<k_{n+1}$, put $X'_k := X_n$. Then the sequence $\xi'=\{ X'_k: k\in\w\}$ $\I$-covers $X$ and, for every $k\in\w$, $X'_k$ is $\I$-covered by $\eta_k$.\qed}
\end{remark}

For the property $\gamma_\I$ we have the following analogue of Proposition \ref{p:property-gamma} (we omit its proof because it is very similar to the proof of Proposition \ref{p:property-gamma}).
\begin{proposition} \label{p:property-gamma-k}
Let $X$ be a topological space and let $\I$ be an ideal of compact sets of $X$. Then $X$ has the property $\gamma_\I$ if and only if it has the properties  $\varphi_\I$ and $\e_\I$.
\end{proposition}



Let $X$ be a Tychonoff space and let $\{X_n\}_{n\in\w}$ be an increasing cover of $X$ (i.e.,  $X=\bigcup_{n\in\w} X_n$ and $X_n\subseteq X_{n+1}$ for all $n\in\w$).
For every $n\in\w$, let $\I_n$ be an ideal of compact subsets in $X_n$ such that $\I_n\subseteq \I_{n+1}$. It is easy to see that the family $\I := \bigcup_{n\in\w} \I_n$ is an ideal of compact subsets of $X$. Conversely, if $\I$ is an ideal of compact subsets of $X$, then for every $n\in\w$, the family $\I_n =\I\cap X_n :=\{ K\in \I: K\subseteq X_n\}$ is an ideal of compact subsets in $X_n$.
We shall say that $\{X_n\}_{n\in\w}$ is {\em $\I$-regular} if $\I := \bigcup_{n\in\w} \I_n$. It is clear that if $\I=\FF(X)$, then any increasing cover $\{X_n\}_{n\in\w}$ of $X$ is $\FF(X)$-regular.

Let $X$ be a Tychonoff space and let $\I$ be an ideal of compact subsets of $X$. For every $n\in\w$, denote by $X_n$ the space $X$ and let $Y:=\bigoplus_{n\in\w} X_n$ be the direct topological sum of the spaces $X_n$. Define the ideal $\I(Y)$ of compact sets in $Y$ as the direct sum of $\I$, i.e., a compact subset $F$ of $Y$ belongs to $\I(Y)$ if and only if $F\cap X_n\in \I$ for every $n\in\w$. Then, by the definition of $\I(Y)$, the cover $\{\bigoplus_{i=0}^n X_i\}_{n\in\w}$ of $Y$  is $\I(Y)$-regular.

We say that the property $\gamma_\I$ is {\em $\sigma$-additive} on a Tychonoff space $X$ if for every $\I$-regular increasing cover $\{X_n\}_{n\in\w}$ of $X$ such that all $X_n$ have the property $\gamma_{\I_n}$ with $\I_n = \I\cap X_n$, it follows that the space $X$ has the property $\gamma_\I$.
In \cite[Corollary~14]{Jordan}, Jordan proved that the property $\gamma$ (=$\gamma_p$) is $\sigma$-additive for every Tychonoff space $X$.
We do not know whether for every Tychonoff space $X$ and each ideal $\I$ of compact sets in $X$, the property $\gamma_\I$ is $\sigma$-additive. Below we consider only one special case used in what follows.

\begin{proposition} \label{p:sum-gamma-I-space}
Let $X$ be a topological space and let $\I$ be an ideal of compact sets of $X$. Then $X$ has the property $\gamma_\I$ if and only if the direct topological sum $Y=\bigoplus \{X_i: X_i=X \mbox{ for each } i\in \omega\}$ has the property $\gamma_{\I(Y)}$.
\end{proposition}

\begin{proof}
Assume that $X$ has the property $\gamma_\I$.
Let $\U=\{U_{\alpha}: \alpha\in \Lambda\}$ be an open $\I(Y)$-cover of $Y$. We have to show that $\U$ contains an $\I(Y)$-sequence. For every $n\in\w$, we identify $X_n$ with $X$ and set
\[
\V_n:= \{ (U_\alpha \cap X_0)\cap\cdots\cap (U_\alpha \cap X_n): \alpha\in\Lambda\},
\]
and observe that $\V_n$ is an $\I$-cover of $X$ (indeed, if $F\in\I$ take $\alpha\in\Lambda$ such that $\bigoplus_{i=0}^n F \subseteq U_\alpha$). Note that $X\in\V_n$ if and only if $\bigoplus_{i=0}^n X_i \subseteq U_\alpha$ for some $\alpha\in \Lambda$. We distinguish between two cases.

{\em Case 1. There exist a sequence $\{\alpha_k\}_{k\in\w}\subseteq \Lambda$ and a strictly increasing sequence $\{n_k\}_{k\in\w}\subseteq \w$ such that $\bigoplus_{i=0}^{n_k} X_i \subseteq U_{\alpha_k}$.} Then the sequence $\{U_{\alpha_k}\}_{k\in\w}$ is an $\I$-sequence in $\U$, and we are done.
\smallskip

{\em Case 2. There is an $m_0\in\w$ such that
\begin{equation} \label{equ:I-cover-1}
X_0 \oplus\cdots\oplus X_n \not\subseteq U_\alpha \; \mbox{ for every } \alpha\in\Lambda \mbox{ and } n\geq m_0.
\end{equation} }
Since $X$ has the property $\gamma_\I$ and all $\V_n$ are $\I$-covers of $X$, for every $n\geq m_0$ there is an $\I$-sequence $\SC_n =\{ V_{\alpha(i,n)}: i\in\w\}  \subseteq \V_n$, where $\alpha(i,n)\in \Lambda$ and
\[
V_{\alpha(i,n)} := (U_{\alpha(i,n)}\cap X_0)\cap\cdots\cap (U_{\alpha(i,n)}\cap X_n) \; \mbox{ for all } \; i\in\w.
\]
It follows from (\ref{equ:I-cover-1}) that
\begin{equation} \label{equ:I-cover-2}
V_{\alpha(i,n)} \not= X \; \mbox{ for all } \; i\in\w \; \mbox{ and } \; n\geq m_0.
\end{equation}

Denote by $\Omega$ the set of all finite sequences $\mathbf{i}=(i_0,\dots,i_s)\in\w^{<\w}$ such that $s=i_0$ and $i_0<\cdots<i_s$. For every $\mathbf{i}=(i_0,\dots,i_s)\in \Omega$, set
\begin{equation} \label{equ:I-cover-3}
W_{\mathbf{i}} := V_{\alpha(i_0,m_0)} \cap\cdots\cap V_{\alpha(i_s,m_0+s)}.
\end{equation}
We claim that the family $\W:= \{ W_{\mathbf{i}}: \mathbf{i}\in \Omega\}$ is an open $\I$-cover of $X$. Indeed, let $F\in\I$. Since $\SC_{m_0}$ is an $\I$-sequence, there is $i_0>0$ such that $F\subseteq
V_{\alpha(i_0,m_0)}$. As all $\SC_n$ are $\I$-sequences, there are $i_1,\dots,i_{i_0}\in\w$ such that $i_0<\cdots<i_{i_0}$ and $F\subseteq V_{\alpha(i_t,m_0+t)}$ for all $t=1,\dots,i_0$. Then, putting $\mathbf{i}:=(i_0,\dots,i_{i_0})$, we obtain $F\subseteq W_{\mathbf{i}}$.

Since $X$ has the property $\gamma_\I$, there is a sequence
\[
\{ \mathbf{i}_k=(i^k_0,\dots,i^k_{s_k}): k\in\w \} \subseteq \Omega
\]
such that $\{ W_{\mathbf{i}_k} : k\in\w\}$ is an $\I$-sequence for $X$. We claim that
\begin{equation} \label{equ:I-cover-4}
i_0^k \to \infty \quad \mbox{ as } \; k\to\infty.
\end{equation}
Indeed, suppose for a contradiction that some subsequence $\{ i^{k_j}_0: j\in\w\}$ is contained in a finite set $\{0,\dots,T\}$. Then, by (\ref{equ:I-cover-3}), the sets $W_{\mathbf{i}_{k_j}}$ are contained in one of the sets $V_{\alpha(0,m_0)}, \dots, V_{\alpha(T,m_0)}$. But since $\{ W_{\mathbf{i}_{k_j}} : j\in\w\}$ is also an $\I$-sequence for $X$, we obtain that the family $\{V_{\alpha(0,m_0)}, \dots, V_{\alpha(T,m_0)}\}$ is an $\I$-cover of $X$, and hence $X=V_{\alpha(i,m_0)}$ for some $i\in\{0,\dots,T\}$. However this contradicts (\ref{equ:I-cover-2}). Thus (\ref{equ:I-cover-4}) holds true.
\smallskip

To finish the proof of the necessity  it is sufficient to show that the sequence $\SC:=\{ U_{\alpha(i_{s_k}^k,m_0+i_0^k)}: k\in\w\}$ is an $\I(Y)$-sequence. To this end, fix an arbitrary ${\tilde F}\in \I(Y)$. As ${\tilde F}$ is compact, there is a $q_0\in\w$ such that ${\tilde F}\cap X_n =\emptyset$ for every $n>q_0$. Set $F:= \bigcup_{i=0}^{q_0} ({\tilde F}\cap X_i)$. Then $F\in\I$. By (\ref{equ:I-cover-4}) and since $\{ W_{\mathbf{i}_k} : k\in\w\}$ is an $\I$-sequence, we can choose an $r\in\w$ such that
\[
m_0 + i_0^k >q_0 \; \mbox{ and } \; F\subseteq W_{\mathbf{i}_{k}}\; \mbox{ for every } \; k\geq r.
\]
Now (\ref{equ:I-cover-3}) implies that for every $k\geq r$ (recall that $s_k=i^k_0$)
\begin{equation} \label{equ:I-cover-5}
F\subseteq  W_{\mathbf{i}_k} \subseteq V_{\alpha(i_{s_k}^k,m_0+i_0^k)} = \Big( U_{\alpha(i_{s_k}^k,m_0+i_0^k)}\cap X_0 \Big) \cap\cdots\cap \Big( U_{\alpha(i_0^k,m_0+i_0^k)}\cap X_{m_0+i_0^k} \Big).
\end{equation}
Since $m_0 + i_0^k >q_0$, (\ref{equ:I-cover-5}) implies that ${\tilde F} \subseteq U_{\alpha(i_{s_k}^k,m_0+i_0^k)}$ for every $k\geq r$. Thus $\SC$ is an  $\I(Y)$-sequence.
\smallskip

Assume now that $Y$ has the property $\gamma_{\I(Y)}$. Let $\U$ be an open $\I$-cover of $X$. We have to find an $\I$-sequence in $\U$. Consider the family $\V:=\{ U\oplus \bigoplus_{i=1}^\infty X_i: U\in \U\}$. Clearly, $\V$ is an open $\I(Y)$-cover of $Y$. By the property $\gamma_{\I(Y)}$, choose an $\I(Y)$-sequence $\{ U_n\oplus \bigoplus_{i=1}^\infty X_i: n\in\w\}$ in $\V$. It is easy to see that the sequence $\{ U_n: n\in\w\}\subseteq \U$ is an $\I$-sequence, as desired.\qed
\end{proof}

\subsection{Separation axioms}
Let $X$ and $Y$ be topological spaces. For every $y\in Y$, define the constant function $\yyy\in Y^X$ by $\yyy(x):=y$ for every $x\in X$.  Each ideal $\I$ of compact subsets of $X$ determines the {\em $\I$-open topology $\tau_\I$} on the power space $Y^X$. A subbase of this topology consists of the sets
\[
[K;U]=\{f\in Y^X :f(K)\subseteq U\},
\]
where $K\in\I$ and $U$ is an open subset of $Y$. In particular, for every $g\in Y^X$, finite subfamily $\FF=\{F_1,\dots,F_n\} \subseteq \I$ and each family $\U=\{ U_1,\dots,U_n\}$ of open subsets of $Y$ such that $g(F_i)\subseteq U_i$ for every $i=1,\dots,n$, the sets of the form
\[
W[g;\FF,\U] :=\big\{ f\in Y^X: f(F_i)\subseteq U_i \mbox{ for all } i=1,\dots,n\big\}.
\]
form  a base of the $\I$-open topology $\tau_\I$ at the function $g$. The space  $C(X,Y)$ endowed with the $\I$-open topology induced from $(Y^X,\tau_\I)$ will be denoted by $C_\I(X,Y)$. So, $\CC(X,Y)=C_\I(X,Y)$ for the ideal $\I=\KK(X)$ of compact subsets of $X$. For the ideal $\I=\F(X)$ of all finite subsets of $X$, the function space $C_\I(X,Y)$ is denoted by $C_p(X,Y)$.
If $Y=G$ is an abelian topological group, then also $(G^X,\tau_\I)$ is an abelian topological group and we put $[F;U]:=\bigcap_{x\in F}W[\mathbf{0};\{x\},U]$. So the sets of the form $[F;U]$, where $F\in\I$ and $U\subseteq G$ is an open neighborhood of $0\in G$, form an open  base of the $\I$-open topology at zero function $\mathbf{0}\in G^X$.

Now we consider some separation axioms.
\begin{definition}[\cite{BG-Baire}] {\em
Let $Y$ be a topological space. A topological space $X$ is called
\begin{enumerate}
\item[$\bullet$]  {\em $Y$-Tychonoff} if for every closed subset $A$ of $X$, point $y_0\in Y$ and function $f:F\to Y$ defined on a finite subset  $F$ of $X\setminus A$, there exists a continuous function ${\bar f}:X\to Y$ such that ${\bar f}{\restriction}_F=f$ and ${\bar f}(A)\subseteq \{ y_0\}$;
\item[$\bullet$] {\em $Y$-normal} if $X$ is a $T_1$-space and for any closed set $F\subseteq X$ and each continuous function $f:F\to Y$ with finite image $f(F)$ there exists a continuous function ${\bar f}:X\to Y$ such that ${\bar f}{\restriction}_F=f$.
\item[$\bullet$] {\em $Y$-dimensional} if $X$ is a $T_1$-space and  for any closed set $F\subseteq X$ and continuous function $f:F\to Y$, there exists a continuous function ${\bar f}:X\to Y$ such that ${\bar f}|_F=f$.    \qed
\end{enumerate}}
\end{definition}
It is easy to see that every $Y$-normal space is $Y$-Tychonoff. Observe also that a topological space is $\IR$-Tychonoff or $\IR$-normal if and only if it is Tychonoff or  normal in the standard sense, respectively.
We will use repeatedly the following assertion.
\begin{proposition}[\cite{BG-Baire}] \label{p:Y-good=>good}
Let $Y$ be a topological space admitting a non-constant continuous function $\chi:Y\to\IR$.
\begin{enumerate}
\item[{\rm (i)}] If $X$ is $Y$-Hausdorff, then $X$ is functionally Hausdorff.
\item[{\rm (ii)}] If $X$ is $Y$-Tychonoff, then $X$ is Tychonoff.
\item[{\rm (iii)}] If $X$ is $Y$-normal, then $X$ is normal.
\end{enumerate}
\end{proposition}

We need the following generalization of $Y$-Tychonoffness.
\begin{definition} {\em
Let $X$ and $Y$ be  topological spaces and let $\I$ be an ideal of compact sets of $X$. Then $X$ is called {\em $Y_\I$-Tychonoff} if for every closed subset $A$ of $X$, point $y_0\in Y$, compact set $F\in\I$ with $F\subseteq X\SM A$, and each continuous function $f:F\to Y$ with finite image $f(F)$ there exists a continuous function ${\bar f}:X\to Y$ such that ${\bar f}{\restriction}_{F}=f$ and ${\bar f}(A)\subseteq \{ y_0\}$.\qed}
\end{definition}
If $\I=\FF(X)$ or $\I=\KK(X)$, we shall say simply that $X$ is $Y_p$-Tychonoff (=$Y$-Tychonoff) or $Y_k$-Tychonoff, respectively. It is clear that
\[
\xymatrix{
{\mbox{$Y_k$-Tychonoff}} \ar@{=>}[r] & {\mbox{$Y_\I$-Tychonoff}}  \ar@{=>}[r] & {\mbox{$Y_p$-Tychonoff}}
}
\]
for every ideal $\I$ of compact sets.


The next lemma is a slight extension of Propositions 2.6.9, 2.6.10, 3.4.4 and 3.4.5 of \cite{Eng}, its proof is straightforward and hence omitted.
\begin{lemma} \label{l:I-infinite-sum}
Let $Y$ be a topological space, $X$ be a $Y_\I$-Tychonoff space for an ideal $\I$ of compact sets of $X$, and let $Z=\bigoplus \{X_i: X_i=X \mbox{ for each } i\in \omega\}$. Then the spaces $C_{\I(Z)}(Z,Y)$, $C_{\I}(X,Y^{\w})$ and the product $C_{\I}(X,Y)^{\w}$ are homeomorphic.
\end{lemma}



\begin{proposition} \label{p:Y-good=>good-I}
Let $Y$ be a topological space admitting a non-constant continuous function $\chi:Y\to\IR$, $X$ be a topological space, and let $\I$ be an  ideal of compact sets of $X$. Then:
\begin{enumerate}
\item[{\rm (i)}] If $X$ is $Y_\I$-Tychonoff, then $X$ is a Tychonoff space.
\item[{\rm (ii)}] If in addition $Y$ is  path-connected, then $X$ is Tychonoff if and only if it is $Y_\I$-Tychonoff.
\end{enumerate}
\end{proposition}

\begin{proof}
(i) Since $X$ is also $Y$-Tychonoff, $X$ is Tychonoff by Proposition \ref{p:Y-good=>good}.

(ii) Taking into account (i) it suffices to prove that if $X$ is Tychonoff, then it is a $Y_k$-Tychonoff space. Fix a closed subset $A$ of $X$, point $y_0\in Y$, compact set $F\subseteq X\SM A$, and continuous function $f:F\to Y$ with finite image $f(F)=\{y_1,\dots,y_n\}$, where $y_1,\dots,y_n$ are pairwise distinct. Since $Y$ is path-connected, there is a continuous function $h:[0,n]\to Y$ such that $h(i)=y_i$ for every $0\le i\le n$. For every $i=1,\dots,n$, set $K_i:=f^{-1}(y_i)$, so $K_i$ is a compact subset of $F$ and hence of $X$. Now \cite[3.11(a)]{GiJ} implies that there is a continuous function $g:X\to [0,n]$ such that $g(A)\subseteq \{0\}$ and $g(K_i)=\{i\}$ for $1\le i\le n$. Then the function $h\circ g:X\to Y$ is a desired continuous extension of $f$.\qed
\end{proof}

The next proposition extends Proposition 2.7 of \cite{BG-Baire}.
\begin{proposition} \label{p:2-Y-I-Tychonoff}
Let $X$ be a $T_1$-space, and let $\I$ be an ideal of compact sets of $X$. Then the following conditions are equivalent:
\begin{enumerate}
\item[{\rm (i)}]  $X$ is $Y_\I$-Tychonoff for any nonempty  $T_1$  space $Y$.
\item[{\rm (ii)}]  $X$ is $Y$-Tychonoff for any nonempty  $T_1$  space $Y$.
\item[{\rm (iii)}]   $X$ is $\mathbf{2}_\I$-Tychonoff.
\item[{\rm (iv)}]   $X$ is $\mathbf{2}$-Tychonoff.
\item[{\rm (v)}]   $X$ is a zero-dimensional space.
\end{enumerate}
\end{proposition}

\begin{proof}
The implications (i)$\Rightarrow$(ii), (ii)$\Rightarrow$(iv), (i)$\Rightarrow$(iii) and (iii)$\Rightarrow$(iv) are trivial. The implication (iv)$\Rightarrow$(v) is proved in Proposition 2.7 of \cite{BG-Baire}. To prove (v)$\Rightarrow$(i), assume that $X$ is zero-dimensional. Being also a  $T_1$ space, $X$ is Tychonoff. Given any nonempty $T_1$-space $Y$, fix a closed subset $A$ of $X$, point $y_0\in Y$ and continuous function $f:F\to Y$ with finite image $\{ y_1,\dots,y_n\}$ defined on a subset $F\in\I$  of $X\setminus A$. For every $i=1,\dots,n$, let $F_i :=f^{-1}(y_i)$, so $F_i\in\I$ and all (compact) sets $F_1,\dots,F_n$ and $A$ are  disjoint. Since $X$ is  zero-dimensional and Tychonoff, for every $i=1,\dots,n$, choose a clopen neighborhood $U_i$ of $F_i$ such that the sets $U_1,\dots,U_n$ and $A$ are disjoint.
Define ${\bar f}:X\to Y$ as follows: ${\bar f}(x)=y_i$ if $x\in U_i$,  and ${\bar f}(x)=y_0$ if $x\in X\setminus \bigcup_{i=1}^n U_i$. It is clear that ${\bar f}$ is a continuous extension of $f$ such that ${\bar f}(A)\subseteq \{ y_0\}$. \qed
\end{proof}

Now  we recall the definition of the Baire topology of a topological space $X$ which will be used essentially in the proof of the main results concerning spaces of Baire functions, see Section \ref{sec:Baire-topology}.

Let $(X,\tau)$ be a topological space. For a function $f\in C(X)$, denote by $Z(f):=f^{-1}(0)$ and $CZ(f):=X\SM Z(f)$ the {\em zero-set} and the {\em cozero-set} of $f$.
It is easy to see that the zero-set $Z(f)$ is a closed $G_\delta$-set of $X$. If $X$ is normal, then every closed $G_\delta$-set of $X$ is the zero-set of some real-valued continuous function on $X$, see \cite[3D.3]{GiJ}.
The {\em Baire topology} $\tau_b$ on $X$ is the topology on the underlying set $X$ having for a basis the family of all zero-sets of $X$. Since the countable intersection of zero-sets is also a zero-set, it follows that the space $X$ endowed with the Baire topology and denoted by $X_{\aleph_0}$ is a $P$-space. Recall that a topological space is called a {\em $P$-space} if the intersection of a countable family of open sets is open.
Let us recall also that the family of $G_\delta$-sets in $X$ forms a base of the topology $\tau_\delta$ on $X$, and the space $X$ with the topology $\tau_\delta$ is called  the {\em $P$-modification of}  $X$ and is denoted by $PX$ or $X_\delta$. Clearly, $PX$ is a $P$-space and $\tau_\delta$ is finer than the Baire topology $\tau_b$. If $X$ is a Tychonoff space, then $X_{\aleph_0}=PX$ and $X_{\aleph_0}$ is a Tychonoff space.

Let $X$ and $Y$ be topological spaces. Analogously to zero-sets and cozero-sets we define
\[
\begin{aligned}
Z_Y & :=\{f^{-1}(F) : F \mbox{ is a closed set in $Y$ and } f\in C(X,Y)\},\\
CZ_Y &:=\{f^{-1}(W) : W \mbox{ is an open set in $Y$ and } f\in C(X,Y)\}.
\end{aligned}
\]

\begin{definition} \label{def:X-Y-modification} {\em
Let $Y$ be a topological space containing at least two points, and let $(X,\tau)$ be a $Y$-Tychonoff space. The family of all countable intersections of elements of $Z_Y$ forms a base of the topology $\tau_Y$ on $X$ finer than $\tau$. The space $X_Y:=(X,\tau_Y)$ is called the {\em $Y_\delta$-modification of } $X$.}
\end{definition}
Observe that if $Y$ is a normal space, then $\tau_Y\leq \tau_b$ (see the proof of (i) of Proposition \ref{p:G-del-modif} below). To obtain the equality $\tau_Y= \tau_b$ we introduce the following class of topological spaces.

\begin{definition} \label{def:X-Y-z-normal} {\em
Let $Y$ be a topological space. A topological space $X$ is called
\begin{enumerate}
\item[$\bullet$] {\em $Y$-$z$-Tychonoff} if $X$ is a Tychonoff space and for every closed subset $A\subseteq X$,  disjoint zero-sets $F_1,\dots,F_n$ in $X$ such that $\bigcup_{i=1}^n F_i \subseteq X\SM A$, and each points $y_0,\dots,y_n\in Y$ there exists a continuous function $f:X\to Y$ such that $f(A)\subseteq \{ y_0\}$ and $f(F_i)=\{ y_i\}$ for every $i=1,\dots,n$;
\item[$\bullet$] {\em $Y$-$z$-normal} if $X$ is a Tychonoff space and for each zero-set $F\subseteq X$ and every continuous function $f:F\to Y$ with finite image $f(F)$ there exists a continuous function ${\bar f}:X\to Y$ such that ${\bar f}{\restriction}_F =f$.
\end{enumerate} }
\end{definition}

\begin{proposition} \label{p:Y-z-Tychonoff-exa}
Let $X$ be a Tychonoff space. Then:
\begin{enumerate}
\item[{\rm (i)}] If $X$ is $Y$-normal, then $X$ is $Y$-$z$-normal.
\item[{\rm (ii)}]  If $Y$ is  a path-connected topological space containing at least two points and $X$ is a normal space, then $X$ is $Y$-$z$-normal.
\item[{\rm (iii)}] If $X$ is zero-dimensional, then $X$ is $Y$-$z$-normal for every topological space $Y$.
\item[{\rm (iv)}]  $X$ is Tychonoff if and only if it is $\IR$-z-Tychonoff.
\item[{\rm (v)}] If $X$ is $Y$-$z$-normal, then $X$ is  $Y$-$z$-Tychonoff.
\item[{\rm (vi)}] If $X$ is $Y$-$z$-Tychonoff, then $X$ is  $Y$-Tychonoff.
\end{enumerate}
\end{proposition}

\begin{proof}
(i) and (v) immediately follow from the corresponding definitions.

(ii) By Proposition \ref{p:Y-good=>good} and  Proposition  2.4 of \cite{BG-Baire}, a topological space $X$ is normal if and only if it is $Y$-normal. Now (i) applies.

(iii) Since $X$ is zero-dimensional, Proposition 2.8 of \cite{BG-Baire} states that $X$ is $Y$-normal for any nonempty topological space $Y$, and (i) applies.

(iv) follows from Proposition 1.5.13 of \cite{Eng}.

(vi) Fix a closed subset $A$ of $X$, point $y_0\in Y$ and function $f:F\to Y$ defined on a finite subset  $F=\{ x_1,\dots,x_n\}$ of $X\setminus A$.
Since $X$ is Tychonoff, there is a continuous function $g:X\to [0,n]$ such that $g(A)\subseteq \{0\}$ and $g(x_i)=i$ for every $i=1,\dots,n$. Then the sets $F_i :=g^{-1} \big( [i-\tfrac{1}{3},i+\tfrac{1}{3}]\big)$ are disjoint zero-sets in $X$ such that $\bigcup_{i=1}^n F_i \subseteq X\SM A$.
Define a continuous function $f': \bigcup_{i=1}^n F_i \to Y$ by $f'{\restriction}_{F_i}=f(x_i)$. Since $X$ is $Y$-$z$-Tychonoff, $f'$ has an extension ${\bar f}\in C(X,Y)$. It is clear that ${\bar f}$ is  a desired extension of $f$.
\qed
\end{proof}

\begin{example} \label{exa:Tych-non-R-z-normal}
There is a first countable $\IR$-z-Tychonoff space which is not  $\IR$-z-normal.
\end{example}

\begin{proof}
Let $L$ be the Niemytzki plane, see Example 1.2.4 of \cite{Eng}. Since $L$ is Tychonoff it is $\IR$-z-Tychonoff by (iv) of Proposition \ref{p:Y-z-Tychonoff-exa}. Denote by $L_1$ the line $y=0$. Then $L_1$ is a discrete subspace of $L$. Example 1.5.9 of \cite{Eng} implies that there are two disjoint subsets $A$ and $B$ of $L_1$ such that for every open subsets $U,V$ of $L$ such that $A\subseteq U$ and $B\subseteq V$ it follows that $U\cap V\not=\emptyset$. It is clear that $L_1$ is the zero-set of the function $(x,y)\mapsto y$. Define $f:L_1 \to \IR$ by $f(A)=\{1\}$, $f(B)=\{ 2\}$, and $f\big(L_1\SM (A\cup B)\big)\subseteq \{ 0\}$. Then, by the choice of $A$ and $B$, the function $f$ cannot be extended to $L$. Thus the Tychonoff space $L$ is not $\IR$-z-normal. \qed
\end{proof}

If $X$ is perfectly normal we can prove more. Recall that a topological space $X$ is called {\em perfectly normal} if it is a normal space and every closed subset of $X$ is a $G_\delta$-set. It is clear that a perfectly normal space has countable pseudocharacter, also we recall that every metrizable space is perfectly normal.

\begin{proposition} \label{p:z-normal-Y-normal}
For any perfectly normal space $X$ the following conditions are equivalent:
\begin{enumerate}
\item[{\rm (i)}]  $X$ is $Y$-$z$-Tychonoff;
\item[{\rm (ii)}]  $X$ is $Y$-normal;
\item[{\rm (iii)}]  $X$ is $Y$-$z$-normal.
\end{enumerate}
\end{proposition}

\begin{proof}
By Proposition \ref{p:Y-z-Tychonoff-exa}, (ii)$\Rightarrow$(iii)$\Rightarrow$(i). Let us prove (i)$\Rightarrow$(ii).

Assume that $X$ is a $Y$-$z$-Tychonoff space.
Let $F$ be a closed subset of $X$ and let $f:F\to Y$ be a continuous function with finite image $f(F)=\{y_1,...,y_n\}$ where all $y_i$ are distinct. Since $X$ is a perfectly normal space and $F$ is closed, the set $F_i=f^{-1}(y_i)$ is a zero-set of $X$ for each $i=1,\dots,n$. It is clearly that $F_i\cap F_j=\emptyset$ for all distinct $i,j\in \{1,\dots,n\}$.  By (i), there
exists a continuous function ${\bar f}:X\to Y$ such that ${\bar f}{\restriction}_F=f$. Thus $X$ is $Y$-normal.\qed
\end{proof}


\section{Fr\'{e}chet--Urysohness and sequentiality in some spaces of continuous functions} \label{sec:FU-sequentiality}


Since we shall consider subspaces $H$ of $C(X,Y)$,  some kind of richness of $H$ will be necessary.
Below we generalize the notion of being a $Y$-Tychonoff space.

\begin{definition} \label{def:relatively-Y-Tych} {\em
Let $X,Y$ be topological spaces, $\I$ be an ideal of compact sets of $X$, $\mathcal{D}$ be a subspace of $Y$, and let $H\subseteq S$ be two subspaces of the power space $Y^X$. Then $H$ is called a {\em relatively $\mathcal{D}_\I$-Tychonoff   subspace of $S$} if for every closed subset $A$ of $X$, compact subset $F\in\I$ contained in $X\SM A$, point $y_0\in \mathcal{D}$ and each $f\in S$  such that $f(F)$ is a finite subset of  $\mathcal{D}$ there is a function ${\bar f}\in H$ such that
$
{\bar f}{\restriction}_F=f{\restriction}_F $  and $ {\bar f}(A)\subseteq \{ y_0\}.
$ \qed }
\end{definition}
In other words, $H$ is a relatively $\mathcal{D}_\I$-Tychonoff   subspace of $S$ if for every closed subset $A$ of $X$, every function $f\in S$ and each $F\in\I$ such that $f(F)$ is a finite subset of $\mathcal{D}$ and $F\subseteq X\SM A$, the restriction $f{\restriction}_F$ can be extended to a function ${\bar f}\in H$ with the additional condition ${\bar f}(A)\subseteq \{ y_0\}$ for some point $y_0\in Y$.
If $\I=\FF(X)$ or $\I=\KK(X)$ we shall say that $H$ is a {\em relatively $\mathcal{D}_p$-} or {\em relatively $\mathcal{D}_k$-Tychonoff subspace of $S$}, respectively. Observe that $X$ is $Y$-Tychonoff if and only if $H=C(X,Y)$ is a relatively $Y_p$-Tychonoff subspace of $S=Y^X$.

Let us recall that a topological space $X$ is called {\em Fr\'{e}chet--Urysohn} if for any cluster point $a\in X$ of a subset $A\subseteq X$ there is a sequence $\{ a_n\}_{n\in\w}\subseteq A$ which converges to $a$.

\begin{proposition} \label{p:FU-Cp}
Let $Y$ be a $T_1$ topological space containing a two element subset $\mathcal{D}=\{ g_0,g_1\}$, and let $X$ be a $Y_\I$-Tychonoff space for some ideal $\I$ of compact sets of $X$. Assume that $H$ is a relatively $\mathcal{D}_\I$-Tychonoff subspace of $C_\I(X,Y)$ containing the constant function $\mathbf{g}_0$. If $H$ is a Fr\'{e}chet--Urysohn space, then $X$ has the property $\gamma_\I$.
\end{proposition}

\begin{proof}
Our proof is based on the idea of the proof of Theorem 4.7.4 of \cite{mcoy}.
Let $\gamma$ be an open $\I$-cover of $X$. If $X\in\gamma$, then $\xi=\{X\}$ is a desired $\I$-sequence. Assume that $X\not\in \gamma$. Set
\[
P:= \big\{f \in H: f^{-1} \big( Y\SM \{g_1\}\big) \subseteq U \mbox{ for some } U\in \gamma\big\}.
\]

We claim that $\mathbf{g}_0\in \overline{P}\setminus P$. Indeed, $\mathbf{g}_0\not\in P$ since $\mathbf{g}_0^{-1} \big( Y\SM \{g_1\}\big)=X\not\in \gamma$. Let $W=W[\mathbf{g}_0;F,V]\cap H$ be a standard neighborhood of $\mathbf{g}_0$, where $F\in \I$ and $V\subseteq Y$ is a neighborhood of $g_0$.  As $\gamma$ is an $\I$-cover, there is $U\in\gamma$ such that $F\subseteq U$. Since $X$ is $Y_\I$-Tychonoff and $H$ is relatively $\mathcal{D}_\I$-Tychonoff, there is a function $f\in H$ such that $f(x)= g_0$ for every $x\in F$ and $f(X\SM U)\subseteq \{ g_1\}$. Then $f^{-1} \big( Y\SM \{g_1\}\big)\subseteq X\SM (X\SM U) =U$. Therefore $f\in P\cap W$. Thus $\mathbf{g}_0\in \overline{P}\setminus P$.

Since $H$ is a Fr\'{e}chet--Urysohn space, there is a sequence $\{ f_n\}_{n\in\w}\subseteq P$ such that $f_n\to \mathbf{g}_0$.
For every $n\in\w$, the choice of $f_n$ implies that there is a $U_n\in\gamma$ such that $f^{-1}_n \big( Y\SM \{g_1\}\big) \subseteq U_n$. Set $\xi:=\{ U_n\}_{n\in\w}$. We show that $\xi$ is an $\I$-sequence. Indeed, fix an arbitrary $F\in \I$ and choose an open neighborhood $V$ of $g_0\in Y$ such that $g_1\not\in V$ (recall that $Y$ is $T_1$). Then  there exists an $m\in\w$ such that $f_n\in W[\mathbf{g}_0; F,V]$ for every $n\geq m$. In particular, this implies that $F\subseteq f^{-1}_n(V)\subseteq f^{-1}_n \big( Y\SM \{g_1\}\big) \subseteq U_n$ for all $n\geq m$. Thus $\xi$ is an $\I$-sequence as desired.\qed
\end{proof}

Below we prove the first main result of this section.
\begin{theorem} \label{t:FU-Cp}
Let $Y$ be a metrizable space containing a two element subset $\mathcal{D}=\{ g_0,g_1\}$, and let $X$ be a $Y_\I$-Tychonoff space for some ideal  $\I$ of compact sets of $X$. Assume that $H$ is a relatively $\mathcal{D}_\I$-Tychonoff subspace of $C_\I(X,Y)$ containing the constant function $\mathbf{g}_0$. Then $H$ is a Fr\'{e}chet--Urysohn space if and only if $X$ has the property $\gamma_\I$.
\end{theorem}

\begin{proof}
If $H$ is Fr\'{e}chet--Urysohn, then $X$ has the property $\gamma_\I$ by Proposition \ref{p:FU-Cp}.

Conversely, assume that $X$ has the property $\gamma_\I$.
It is well known that any metrizable space $Y$ can be isometrically embedded into some Banach space $E$. Therefore $H$ is a subspace of $C_\I(X,E)$. Hence it suffices to prove that the space $C_\I(X,E)$ is Fr\'{e}chet--Urysohn. Let $A\subseteq C_\I(X,E)$ and $f\in \overline{A}\SM A$. Since $C_\I(X,E)$ being a locally convex space is homogenous, we can assume that $f=\mathbf{0}$ is the zero-function. Let $\{ V_n\}_{n\in\w}$ be a strictly decreasing open base at $0$ in $E$.

Below we use the idea from the proof of Theorem 4.7.4 of \cite{mcoy}. For every $n\in\w$ and each $F\in\I$, choose a continuous function $g_{n,F}\in  [F;V_n]\cap A$, and set
\[
W(n,F):=\{ x\in X: g_{n,F}(x)\in V_n\}\; \mbox{ and } \; \mathcal{W}_n :=\{ W(n,F): F \in \I \}.
\]
Since $\mathbf{0}\in \overline{A}\SM A$, it follows that $\mathcal{W}_n$ is an open  $\I$-cover of $X$ for every $n\in\w$.
Now we define a sequence $\{\U_n\}_{n\in\w}$ of open  $\I$-covers of $X$ as follows: $\U_0 :=\mathcal{W}_0$ and $\U_n :=\U_{n-1} \wedge \mathcal{W}_n$ for $n\geq 1$.

It can be assumed that $X\not\in\I$ (otherwise, the space $C_\I(X,E)$ is metrizable). Then the family $\big\{ X\SM \{x\}: x\in X\big\}$ is an open $\I$-cover of $X$. Therefore, by hypothesis, there is a sequence $\{ x_n\}_{n\in\w}$ in $X$ such that $\big\{ X\SM \{x_n\}: n\in \w\big\}$ is an $\I$-sequence. For each $n\in\w$, define $\U'_n :=\{ U\SM \{x_n\}: U\in \U_n\}$, and let $\mathcal{V}:=\bigcup_{n\in\w} \U'_n$. Then $\mathcal{V}$ is an  open $\I$-cover  of $X$. Now choose an $\I$-sequence $\{ O_k\}_{k\in\w}$ from $\mathcal{V}$.

For every $k\in\w$, choose an $n_k\in\w$ such that $O_k \subseteq U_{n_k}\SM \{ x_{n_k}\}$ for some $U_{n_k}\in \U_{n_k}$. So there is an $F_{n_k}\in\I$ such that $O_k \subseteq W(n_k, F_{n_k}) \SM \{ x_{n_k}\}$. Observe that the sequence $\{ n_k\}_{k\in\w}$ cannot be bounded (indeed, if $n_k\leq M$ for all $k\in\w$, then the finite set $\{x_0,\dots,x_M\}\in \I$ is not contained in $O_k$ for every $k\in\w$, and hence $\{ O_k\}_{k\in\w}$  is not an $\I$-sequence, a contradiction). Take an increasing subsequence $\{ n_{k_i}\}_{i\in\w}$, and let $f_i := g_{n_{k_i},F_{k_i}}$.

We claim that $f_i\to \mathbf{0}$. Indeed, take an $F\in\I$ and an open neighborhood $V$ of $0\in E$. Choose an $m\in\w$ such that $V_{n_{k_i}}\subseteq V$ and $F\subseteq O_{k_i}$ for every $i\geq m$. Then,  for every $i\geq m$, $f_i(F)\subseteq V_{n_{k_i}}\subseteq V$ and hence  $f_i\in [F;V]$. Thus $f_i\to \mathbf{0}$. \qed
\end{proof}
The metrizability condition on $Y$ in Theorem \ref{t:FU-Cp} is essential. Indeed, if $X$ is a singleton and $Y=\{ 0,1\}^{\w_1}$, then $C_p(X,Y)$ is topologically isomorphic to the compact non-sequential space $Y$.

The next theorem gives a complete answer to Problem \ref{prob:gen-top-C(X)} for the property of being a Fr\'{e}chet--Urysohn space.
\begin{theorem} \label{t:Cp(X,Y)-FU}
Let $Y$ be a metrizable space containing at least two points, and let $X$ be a $Y_\I$-Tychonoff space for some ideal $\I$ of compact sets of $X$. Then the following statements are equivalent:
\begin{enumerate}
\item[{\rm (i)}] $C_\I(X,Y)$ is a Fr\'{e}chet--Urysohn space;
\item[{\rm (ii)}] $C_\I(X,Y)^\w$ is a Fr\'{e}chet--Urysohn space;
\item[{\rm (iii)}] $C_\I(X,Y^\w)$ is a Fr\'{e}chet--Urysohn space;
\item[{\rm (iv)}] $X$ has the property $\gamma_\I$.
\end{enumerate}
\end{theorem}

\begin{proof}
The equivalence (i)$\Leftrightarrow$(iv) follows from Theorem \ref{t:FU-Cp}. The implication (ii)$\Rightarrow$(i) is clear since $C_\I(X,Y)$ is homeomorphic to a subspace of $C_\I(X,Y)^\w$, and the equivalence (ii)$\Leftrightarrow$(iii) follows from Lemma \ref{l:I-infinite-sum}. To prove (iv)$\Rightarrow$(ii) we note first that, by Lemma \ref{l:I-infinite-sum},  $C_\I(X,Y)^\w$ is homeomorphic to the space  $C_{\I(Z)}(Z,Y)$ where $Z=\bigoplus \{X_i: X_i=X \mbox{ for each } i\in \omega\}$. Now Proposition \ref{p:sum-gamma-I-space} implies that the space $Z$ has the property $\gamma_{\I(Z)}$. Finally, applying Theorem \ref{t:FU-Cp} we obtain that the space $C_\I(X,Y)^\w$ is Fr\'{e}chet--Urysohn. \qed
\end{proof}

Let $Y=\IR$ or $Y=\II:=[0,1]$. Then, by Theorem 3.1.7 of \cite{Eng}, $X$ is $Y_\I$-Tychonoff if and only if it is a Tychonoff space. Now Theorem \ref{t:Cp(X,Y)-FU} implies the following extension of Theorem 4.7.4 of \cite{mcoy}:


\begin{corollary} \label{c:Cp(X)-FU}
Let $X$ be a Tychonoff space, and let $\I$ be an ideal of compact sets of $X$. Then the following statements are equivalent:
\begin{enumerate}
\item[{\rm (i)}] $C_\I(X)$ is a Fr\'{e}chet--Urysohn space;
\item[{\rm (ii)}] $C_\I(X)^\w$ is a Fr\'{e}chet--Urysohn space;
\item[{\rm (iii)}] $C_\I(X,\II)$ is a Fr\'{e}chet--Urysohn space;
\item[{\rm (iv)}] $C_\I(X,\II)^\w$ is a Fr\'{e}chet--Urysohn space;
\item[{\rm (v)}] $C_\I(X,\II^\w)$ is a Fr\'{e}chet--Urysohn space;
\item[{\rm (vi)}] $X$ has the property $\gamma_\I$.
\end{enumerate}
\end{corollary}

  Proposition \ref{p:2-Y-I-Tychonoff} and  Theorem \ref{t:Cp(X,Y)-FU} imply
\begin{corollary} \label{c:Ck(X,2)-FU}
Let $X$ be a zero-dimensional $T_1$-space  and let $\I$ be an ideal of compact sets of $X$. Then  the following statements are equivalent:
\begin{enumerate}
\item[{\rm (i)}] $C_\I(X,\mathbf{2})$ is a Fr\'{e}chet--Urysohn space;
\item[{\rm (ii)}] $C_\I(X,\mathbf{2})^\w$ is a Fr\'{e}chet--Urysohn space;
\item[{\rm (iii)}] $C_\I(X,\mathbf{2}^\w)$ is a Fr\'{e}chet--Urysohn space;
\item[{\rm (iv)}] $X$ has the property $\gamma_\I$.
\end{enumerate}
\end{corollary}

\begin{remark} \label{rem:gamma-k}{\em
Combining Theorem \ref{t:Ascoli-C(X,2)-A} with Corollary \ref{c:Ck(X,2)-FU} we obtain that a zero-dimensional metric space $X$ has the property $\gamma_k$ if and only if $X$ is a Polish locally compact space. It follows from \cite[Theorem~9.3]{Osipov-18} that if $X$ is an arbitrary separable metric space,  then the space $X$ has the property $\gamma_k$ if and only if $X$ is hemicompact. \qed}
\end{remark}

Recall that a topological space $X$ is called {\em sequential} if for each non-closed subset $A\subseteq X$ there is a sequence $\{a_n\}_{n\in\w}\subseteq A$ converging to some point $a\in \bar A\setminus A$.
Below we characterize  sequentiality of $\C_\I(X,\mathbf{2})$.
We say that a cover $\gamma$ of a topological space $X$ is {\em clopen} if every $U\in \gamma$ is a clopen (=closed and open) subset of $X$.
Recall (see Proposition \ref{p:2-Y-I-Tychonoff}) that a $T_1$-space $X$ is zero-dimensional if and only if it is $\mathbf{2}$-Tychonoff.

\begin{theorem} \label{t:CI-2-sequential}
Let $X$ be a  zero-dimensional $T_1$-space, and let $\I$ be an ideal of compact subsets of $X$. Then the space $C_\I(X,\mathbf{2})$ is sequential if and only if for every clopen $\I$-cover $\gamma$ of $X$ either $X\in\gamma$ or there is a sequence $\{ U_n\}_{n\in\w}$ in $\gamma$ such that
\begin{enumerate}
\item[{\rm (i)}] the set $W_0:=\lim_n U_n$ is a clopen subset of $X$ and $W_0\not\in \gamma$;
\item[{\rm (ii)}] the set $W_0$ is $\I$-covered by $\{  U_n\}_{n\in\w}$;
\item[{\rm (iii)}] the set $X\SM W_0$ is $\I$-covered by $\{ X\SM U_n\}_{n\in\w}$.
\end{enumerate}
\end{theorem}

\begin{proof}
Assume that the space $C_\I(X,\mathbf{2})$ is sequential and let $\gamma$ be a clopen $\I$-cover of $X$. If $X\in\gamma$, we are done.
Assume now that $X\not\in \gamma$. Since $\gamma$ is clopen, for every $U\in\gamma$ we can define a function $f_U\in C_\I(X,\mathbf{2})$ setting
\[
f_U^{-1}(0):=U \; \mbox{ and } \; f_U^{-1}(1):=X\SM U.
\]
Set $P:=\{ f_U: U\in\gamma\}$ and denote by $\mathbf{0}\in C_\I(X,\mathbf{2})$ the zero function.  We show that $\mathbf{0}$ belongs to $\overline{P}\setminus P$. Indeed, $\mathbf{0}\not\in P$ since $\mathbf{0}^{-1} \big( 0\big)=X\not\in \gamma$. Let $W=W[\mathbf{0};F,\{0\}]$ be a standard neighborhood of $\mathbf{0}$, where $F\in \I$.  As $\gamma$ is an $\I$-cover, there is $U\in\gamma$ such that $F\subseteq U$. Then clearly $f_U\in W$. Thus $\mathbf{0}\in \overline{P}\setminus P$, and hence $P$ is a non-closed subset of $C_\I(X,\mathbf{2})$.

Since $C_\I(X,\mathbf{2})$ is a sequential space, there is a sequence $\{ f_{U_n}\}_{n\in\w}\subseteq P$ converging to some function  $h\in \overline{P}\setminus P$. Set $W_0:=h^{-1}(0)$ and observe that $W_0$ is a  clopen subset of $X$. The equality $h=\lim_n f_{U_n}$ and the discreteness of the space $\mathbf{2}$ easily imply that $W_0=\lim_n U_n$ and $h^{-1}(1)= X\SM W_0=\lim_n X\SM U_n$. To prove (i), we have to show that $W_0\not\in \gamma$. But if $W_0\in \gamma$ we would have $h=f_{W_0}\in P$ that contradicts the choice of $h$.

To show (ii) and (iii), fix a compact set $F\in\I$ such that $F\subseteq W_0$ (respectively, $F\subseteq X\SM W_0$). Since $f_{U_n}\to h$, there is an $m\in\w$ such that $f_{U_n} \in W[h;F,\{0\}]$  (respectively, $f_{U_n} \in W[h;F,\{1\}]$) for every $n\geq m$. But this means that $F\subseteq f_{U_n}^{-1}(0)=U_n$ (respectively, $F\subseteq f_{U_n}^{-1}(1)=X\SM U_n$)  for every $n\geq m$. Thus (ii) and (iii) hold true.
\smallskip

Conversely, assume that for every clopen $\I$-cover $\gamma$ of $X$ either $X\in\gamma$ or there is a sequence $\{ U_n\}_{n\in\w}$ in $\gamma$ satisfying (i)-(iii).
To show that the space $C_\I(X,\mathbf{2})$ is sequential, for every non-closed subset $A$ of $C_\I(X,\mathbf{2})$, we have to find a function $f\in \overline{A}\SM A$ and a sequence $\{ f_n\}_{n\in\w}$ in $A$ converging to $f$. Since $C_\I(X,\mathbf{2})$ being an abelian topological group is homogenous, we can assume that the zero-function $\mathbf{0}$ belongs to $\overline{A}\SM A$.

Set
\[
\mathcal{W} :=\{ g^{-1}(0): g \in A \}.
\]
Since $\mathbf{0}\in \overline{A}\SM A$ and all $g^{-1}(0)$ are clopen, it follows that $\mathcal{W}$ is a clopen  $\I$-cover of $X$. Observe that $X\not\in \mathcal{W}$ since, otherwise, we would have $g^{-1}_0(0)=X$ for some $g_0\in A$, and hence $\mathbf{0}=g_0\in A$, a contradiction. By assumption, there is a sequence $\{ f_n\}_{n\in\w}$ in $A$ such that the set
\[
W_0 =\lim_n f_n^{-1}(0) \not\in \gamma
\]
is a clopen subset of $X$ and is $\I$-covered by $\{  f_n^{-1}(0)\}_{n\in\w}$, and the set $X\SM W_0$ is $\I$-covered by the sequence $\{ X\SM f_n^{-1}(0)\}_{n\in\w}$. 
 Since $W_0$ is clopen,  we can define $f\in C_\I(X,\mathbf{2})$ by $f^{-1}(0)=W_0$ and $f^{-1}(1)=X\SM W_0$. Let us show that $f\not\in A$. Indeed, otherwise, the set $W_0=f^{-1}(0)$ belongs to $\mathcal{W}$ that contradicts (i).

We claim that $f_n\to f$. Indeed, since $W_0$ is clopen and the range is the doubleton $\mathbf{2}$, for every neighborhood $\mathcal{O}$ of $f$ there are $F_0,F_1\in\I$ such that $F_0 \subseteq W_0$, $F_1\subseteq X\SM W_0$ and $\widetilde{\mathcal{O}} :=W[f;F_0, \{0\}] \cap W[f; F_1, \{1\}]\subseteq \mathcal{O}$. By (ii) and (iii), choose an $m\in\w$ such that $F_0 \subseteq f_n^{-1}(0)$ and $F_1 \subseteq X\SM f_n^{-1}(0)=f_n^{-1}(1)$ for every $n\geq m$. Then, for every $x\in F_0\cup F_1$ and each $n\geq m$, we have $f_n(x)=f(x)$ and hence $f_n \in \widetilde{\mathcal{O}} \subseteq \mathcal{O}$.  Thus $f_n\to f$. \qed
\end{proof}

\begin{remark} \label{rem:gamma-sequential} {\em
Let $X$ be a zero-dimensional $T_1$-space. If $X$ has the property $\gamma_\I$, then $X$ satisfies the conditions of Theorem \ref{t:CI-2-sequential}. Indeed, let $\mu$  be a clopen $\I$-cover of $X$. If $X\in\mu$, we are done. If $X\not\in\mu$, then, by the property $\gamma_\I$, there is an $\I$-sequence $\{ U_n\}_{n\in\w}$ in $\mu$. Thus the conditions (i)-(iii) of Theorem \ref{t:CI-2-sequential} are satisfied if we put $W_0=X$. \qed}
\end{remark}

\begin{remark} \label{rem:sequential-product} {\em
By Corollary \ref{c:Ck(X,2)-FU}, if $C_\I(X,\mathbf{2})$ is  Fr\'{e}chet--Urysohn, then also $C_\I(X,\mathbf{2})^\w$ is a Fr\'{e}chet--Urysohn space. One can ask whether the same is true for sequentiality. In general the answer is ``no''. Indeed, let $X$ be a countable non-locally compact metric space with one non-isolated point. Then, by Theorem \ref{t:Ascoli-C(X,2)-A}, the space $\CC(X,\mathbf{2})$ is sequential (and non-Fr\'{e}chet--Urysohn). However, the space $\CC(X,\mathbf{2})^\w =\CC(Y,\mathbf{2})$, where $Y=\bigoplus_\w X$, is not sequential by the same Theorem \ref{t:Ascoli-C(X,2)-A}.\qed}
\end{remark}

We finish this section with the following problem.
\begin{problem} \label{prob:C(X,2)-sequential}
Let $X$ be a zero-dimensional (metric) $T_1$-space such that the space $C_p(X,\mathbf{2})$ is sequential. $(\alpha)$ Is it true that $C_p(X,\mathbf{2})$ is Fr\'{e}chet--Urysohn? $(\beta)$ Is it true that $C_p(X,\mathbf{2})^\w$ is sequential?
\end{problem}

Of course, by Corollary \ref{c:Ck(X,2)-FU},  a positive answer to $(\alpha)$ of Problem \ref{prob:C(X,2)-sequential} implies a positive answer to $(\beta)$ of this problem. However, we conjecture that answers to both questions in Problem \ref{prob:C(X,2)-sequential} are negative.


\section{The $k$-space property, normality and some cardinal numbers  for subspaces of $C_\I(X,Y)$} \label{sec:k-normal}


Recall that a topological space $X$ is called  a {\em $k$-space} if for each non-closed subset $A\subseteq X$ there is a compact subset $K\subseteq X$ such that $A\cap K$ is not closed in $K$.
To consider the $k$-space property for subspaces of $C_\I(X,Y)$ we need two lemmas.

\begin{lemma} \label{l:Cp-k-space}
Let $Y$ be a topological space containing a discrete and closed sequence $\mathcal{D}=\{ g_n\}_{n\in\w} \subseteq Y$, and let $X$ be a $Y_\I$-Tychonoff space for some  ideal $\I$ of compact sets of $X$. Assume that $H$ is a relatively $\mathcal{D}_\I$-Tychonoff  subspace of $C_\I(X,Y)$ containing the constant function $\mathbf{g}_0$. If $H$ is a $k$-space, then $X$ has the property $\varphi_\I$.
\end{lemma}

\begin{proof}
We use the idea of the proof of Theorem 3 in \cite{Gerlits}.
Suppose for a contradiction that $X$ does not have the property $\varphi_\I$. Then there exists an open $\I$-cover $\eta =\{ \eta_n: n\in\w\}$ of $X$ with $\eta_n\leq \eta_{n+1}$ such that for every $\I$-sequence  $\xi=\{ X_n: n\in\w\}$, one can find an $m=m_\xi \in\w$ such that  $X_m$ is not $\I$-covered by $\eta_m$. For every $n\in\w$, set
\[
A_n :=\big\{ f\in H: f^{-1}\big( Y\SM \{ g_k: k>n\}\big) \mbox{ is $\I$-covered by } \eta_n\big\},
\]
and put $A:=\bigcup_{n\in\w} A_n$.

We claim that the constant function $\mathbf{g}_0$ belongs to $\overline{A}\SM A$ and hence $A$ is not closed in $H$. Indeed, first we observe that $\mathbf{g}_0\not\in A$ since $\mathbf{g}_0^{-1}\big( Y\SM \{ g_k: k>n\}\big)=X$ is not $\I$-covered by $\eta_n$ for every $n\in\w$ (if $X$ is $\I$-covered by $\eta_m$ for some $m\in\w$, then $X$ is $\I$-covered by $\eta_n$ for all $n\geq m$ since $\eta_m\leq \eta_n$; but this contradicts the choice of the cover $\eta$, see also Remark \ref{r:varphi-I}).
Let now $W=W[\mathbf{g}_0; F,\mathcal{O}]\cap H$ be a standard neighborhood of $\mathbf{g}_0$, where $F\in \I$ and $\mathcal{O}$ is an open neighborhood of $g_0\in Y$. Since $\eta=\{ \eta_n: n\in\w\}$ is an open $\I$-cover, there are an $l\in \w$ and an open set $U\in \eta_l$ such that $F\subseteq U$. As $X$ is $Y_\I$-Tychonoff and $H$ is a relatively $\mathcal{D}_\I$-Tychonoff subspace of $C_\I(X,Y)$, there exists a function $f\in H$ such that $f(F)=\{ g_0\}$ and $f(X\SM U)=\{ g_{l+1}\}$. It is clear that $f\in A_l$ and $f\in W$. Thus $\mathbf{g}_0\in \overline{A}\SM A$.

Let us show that for every $n\in\w$, $A_n$ is closed in $H$. Indeed, fix an arbitrary $f\in H\SM A_n$. Then the set $B_n:=f^{-1}\big( Y\SM \{ g_k: k>n\}\big)$ is not $\I$-covered by $\eta_n$. Hence there exists an $F\in\I$ with $F\subseteq B_n$ such that no member of $\eta_n$ contains $F$. Set $U_F := Y\SM \{ g_k: k>n\}$ (so $U_F$ is open  because the sequence $\{ g_n\}_{n\in\w}$ is closed in $Y$) and put
\[
W:= W[f; F, U_F]\cap H .
\]
Then $W$ is an open neighborhood of $f$.  Now, if $g\in W\cap A_n$, then $F\subseteq g^{-1}(U_F)=g^{-1}(Y\SM \{ g_k: k>n\})$ is $\I$-covered by $\eta_n$ that contradicts the choice of the set $F$.  Thus  $W\cap A_n=\emptyset$ and hence $A_n$ is closed in $H$.

Since $H$ is a $k$-space, there is a compact subset $K$ of $H$ such that $K\cap A$ is not closed in $K$. As $K$ is compact, for every $x\in X$, the set $\{ f(x):f\in K\}$ is compact in $Y$. Taking into account that the sequence $\{ g_n\}_{n\in\w}$ is closed and discrete in $Y$, it follows that  for every $x\in X$, there is an $n(x)\in\w$ such that
\[
f(x) \in Y\SM \{ g_k: k>n(x)\} \quad \mbox{ for every }\; f\in K.
\]
For  every $n\in\w$, set $X_n :=\{ x\in X: n(x)\leq n\}$. Then $X_n\subseteq X_{n+1}$ and $X=\bigcup_{n\in\w} X_n$. We show that the sequence $\xi=\{X_n:n\in\w\}$ is also an $\I$-sequence.
Indeed, let $C\in\I$. Then $C$ is a compact subset of $X$ and $C_\I(C,Y)=\CC(C,Y)$ because $\I$ is an ideal. Now Theorems 2.6.11 and 3.4.3 of \cite{Eng} imply that the evaluation map
\[
\Omega:\CC(C,Y)\times C \to Y, \quad \Omega(f,x):= f(x),
\]
is continuous. Therefore the image  $\Omega(K,C)$ is a compact  subset of $Y$. Since $\mathcal{D}$ is closed and discrete, there is an $r\in\w$ such that $n(x)\leq r$ for every $x\in C$, i.e. $C\subseteq X_r$. Thus $\xi$ is an $\I$-sequence as stated.

The choice of the open $\I$-cover $\eta =\{ \eta_n: n\in\w\}$ and Remark \ref{r:varphi-I} imply that there exists an $m\in\w$ such that $X_m$ is not $\I$-covered by $\eta_n$ for every $n\in\w$.

We show that $K\cap A_k=\emptyset$ for every natural number $k>m$. Indeed, fix $k>m$ and let $f\in A_k$. By the definition of $A_k$, the set $f^{-1}\big( Y\SM \{ g_i: i>k\}\big)$  is $\I$-covered by $\eta_k$. Since $X_m$ is not $\I$-covered by $\eta_k$, we obtain that
\[
X_m \SM f^{-1}\big( Y\SM \{ g_i: i>k\}\big) \not=\emptyset,
\]
i.e. there is a point $x\in X_m$ such that $f(x)\in \{ g_i: i>k>m\}$. Then the definition of $X_m$ implies that $f\not\in K$. Thus $K\cap A_k=\emptyset$.

Now $K\cap A=K\cap \bigcup_{k\in\w} A_k = \cup_{k=0}^m (K\cap A_k)$ is closed in $K$ since all $A_n$s are closed in $H$. But this contradicts the choice of $K$. Thus the space $X$ must have the property $\varphi_\I$.\qed
\end{proof}

\begin{lemma} \label{l:H-k-space-gamma}
Let $Y$ be a topological space containing a discrete two-element subset $\mathcal{D}=\{ g_0, g_1\}$, $X$ be a Tychonoff space satisfying the property $\varphi$, and let $H$ be a subspace of $C_p(X,Y)$ such that $H\cap C_p(X, \mathcal{D})$ is a relatively $\mathcal{D}_p$-Tychonoff subspace of $C_p(X, \mathcal{D})$ and $\mathbf{g}_0\in H$. If $H$ is a $k$-space, then $X$ satisfies the property $\gamma$.
\end{lemma}

\begin{proof}
Since $X$ has the property $\varphi$, Corollary II.3.6 of \cite{Arhangel} implies that $X$ is zero-dimensional.
Observe that the space  $C_p(X, \mathcal{D})$ is a closed subspace of $C_p(X,Y)$,  and hence $H\cap C_p(X, \mathcal{D})$ is closed in $H$. Therefore $H\cap C_p(X, \mathcal{D})$ is  a $k$-space as well.
Since $X$ has the property $\varphi$,  Proposition \ref{p:property-gamma} implies that to prove the lemma it suffices to show that $X$ has the property $\e$.

Suppose for a contradiction that the property $\e$ is not fulfilled. Then there is an open $\w$-cover $\eta$ of $X$ which does not have a countable $\w$-subcover. Put
\[
A:= \big\{ f\in H\cap C(X,\mathcal{D}): f^{-1}(g_0) \mbox{ can be $\w$-covered by a countable subfamily of } \eta \big\}.
\]
Now repeating word for word the proof of Theorem 4 of \cite{Gerlits}, 
 we obtain that $H\cap C_p(X, \mathcal{D})$ is not a $k$-space, a contradiction. Thus $X$ has also the property $\e$. \qed
\end{proof}

Now we prove the first main result of this section, which applying to $H=C_p(X)$ immediately yields Gerlits--Nagy--Pytkeev Theorem \ref{t:GNP}.
\begin{theorem} \label{t:H-k-space}
Let $Y$ be a non-compact metrizable space, and let $X$ be a $Y$-Tychonoff space. Fix a closed and discrete sequence $\mathcal{D}=\{ g_n\}_{n\in\w}$ in $Y$, and let  $H$ be  a relatively $\mathcal{D}_p$-Tychonoff  subspace of $C_p(X,Y)$ containing the constant function $\mathbf{g}_0$. Then the following assertions are equivalent:
\begin{enumerate}
\item[{\rm (i)}] $H$ is a Fr\'{e}chet--Urysohn space;
\item[{\rm (ii)}] $H$ is a sequential space;
\item[{\rm (iii)}] $H$ is a  $k$-space;
\item[{\rm (iv)}] $X$ satisfies the property $\gamma$.
\end{enumerate}
\end{theorem}

\begin{proof}
The implications (i)$\Rightarrow$(ii)$\Rightarrow$(iii) are clear. Observe that  $H\cap C_p(X, \{ g_0,g_1\})$ is a relatively  $\{ g_0,g_1\}_p$-Tychonoff subspace of $C_p(X, \{ g_0,g_1\})$.
Now the implication (iii)$\Rightarrow$(iv) follows from Lemmas \ref{l:Cp-k-space} and  \ref{l:H-k-space-gamma}. Finally, the implication (iv)$\Rightarrow$(i) follows from Theorem \ref{t:FU-Cp}. \qed
\end{proof}

Let $X$ be a Tychonoff space and let $E$ be a locally convex space. Recall that we denote by $C^b(X,E)$ and $C^{rc}(X,E)$ the spaces of all functions $f\in C(X,E)$ such that $f(X)$ is a bounded or a relatively compact subset of $E$, respectively.
We need also the next assertion.
\begin{proposition} \label{p:Y-LCS=>good-I}
Let $E$ be a non-trivial locally convex space, $X$ be a topological space, and let $\I$ be an  ideal of compact sets of $X$. Then $C^b_\I(X,E)$ and $C^{rc}_\I(X,E)$ are  relatively $E_\I$-Tychonoff subspaces of $C_\I(X,E)$.
\end{proposition}

\begin{proof}
The proof follows from the proof of (ii) of Proposition \ref{p:Y-good=>good-I}. Indeed, the extension $h\circ g$ constructed there belongs to  $C^{rc}_\I(X,E)$ because $h\circ g(X)$ is contained in the compact subset $h([0,n])$ of $E$.\qed
\end{proof}

\begin{theorem} \label{t:Pytkeev-Ck-Cb}
Let $X$ be a Tychonoff space, $\I$ be an ideal of compact subsets of $X$, $Y$ be  a metrizable locally convex space, and let $E=C_\I(X,Y)$, $E=C^b_\I(X,Y)$ or $E=C^{rc}_\I(X,Y)$. Then the following assertions are equivalent:
\begin{enumerate}
\item[{\rm (i)}] $E$ is a Fr\'{e}chet--Urysohn space;
\item[{\rm (ii)}] $E$ is a sequential space;
\item[{\rm (iii)}] $E$ is a $k$-space;
\item[{\rm (iv)}] $X$ has the property $\gamma_\I$;
\item[{\rm (v)}] $E^\lambda$ is Fr\'{e}chet--Urysohn (sequential or a $k$-space) for some $\lambda\in (0,\omega]$.
\item[{\rm (vi)}] $E^\lambda$ is Fr\'{e}chet--Urysohn (sequential or a $k$-space) for every $\lambda\in (0,\omega]$.
\end{enumerate}
\end{theorem}

\begin{proof}
The implications (i)$\Rightarrow$(ii)$\Rightarrow$(iii) are clear, and the equivalence (i)$\Leftrightarrow$(iv) follows from Proposition \ref{p:Y-LCS=>good-I} and Theorem \ref{t:FU-Cp}.

(iii)$\Rightarrow$(i) It is well known that there is a linear subspace $Y'$ of $Y$ such that $Y=\mathbb{F}\oplus Y'$, where $\mathbb{F}$ is the field of $Y$. Therefore $E=C_\I(X,Y')\oplus C_\I(X,\mathbb{F})$  ($E=C^b_\I(X,Y')\oplus C^b_\I(X,\mathbb{F})$ or $E=C^{rc}_\I(X,Y')\oplus C^{rc}_\I(X,\mathbb{F})$). Then the space $C_\I(X,\mathbb{F})$ ($C^b_\I(X,\mathbb{F})$ or $C^{rc}_\I(X,\mathbb{F})$) being a closed subspace of $E$ must be a $k$-space. If $\mathbb{F}=\mathbb{C}$, then $C_\I(X,\mathbb{F})=C_\I(X)\oplus C_\I(X)$ and  $C^b_\I(X,\mathbb{F})=C^b_\I(X)\oplus C^b_\I(X)$, and hence the space $C_\I(X)$ (or $C^b_\I(X)$) is a  $k$-space. Now Pytkeev's Theorem \ref{t:Pytkeev-Ck} implies that  the space $C_\I(X)$ (or $C^b_\I(X)$) is Fr\'{e}chet--Urysohn. Hence, by Proposition \ref{p:Y-LCS=>good-I} and Theorem \ref{t:FU-Cp}, the space $X$ has the property $\gamma_\I$. Once again applying Proposition \ref{p:Y-LCS=>good-I} and Theorem \ref{t:FU-Cp}, we obtain that the space $E$ is Fr\'{e}chet--Urysohn.

(iv)$\Rightarrow$(v) By Lemma \ref{l:I-infinite-sum}, the space $C_\I(X,Y)^\lambda$ is homeomorphic to  $C_{\I(Z)}(Z,Y)$, where $Z=\bigoplus \{X_i: X_i=X \mbox{ for each } i\in \lambda\}$. Proposition \ref{p:sum-gamma-I-space} implies that the space $Z$ has the property $\gamma_{\I(Z)}$. Now the equivalence (i)$\Leftrightarrow$(iv) implies that $C_\I(X,Y)^\lambda$ is a Fr\'{e}chet--Urysohn space. Therefore, the subspaces $C^b_\I(X,Y)^\lambda$ and $C^{rc}_\I(X,Y)^\lambda$ of $C_\I(X,Y)^\lambda$  are also Fr\'{e}chet--Urysohn.

The implication (v)$\Rightarrow$(vi) is clear, and (vi) implies (i) (respectively, (ii) or (iii)) because $E$ is a closed subspace of $E^\lambda$. \qed
\end{proof}

Now we consider the tightness $t(H)$ of a subspace $H$  of $C_\I(X,Y)$. Recall that the {\em tightness $t(x,X)$ at a point $x\in X$} of a topological space $X$ is the least infinite cardinality $\kappa$ such that if $x$ is in the closure of a subset $A$ of $X$, then $A$ contains a subset $B$ of cardinality $\leq \kappa$ with $x\in \overline{B}$; the {\em tightness $t(X)$ of $X$} is  the least upper bound of $\{ t(x,X): x\in X\}$.

Let $X$ be a topological space and let $\I$ be an ideal of compact sets in $X$.
Define the {\em $\I$-Lindel\"{o}f degree $\I\mbox{-}\Lin(X)$ of $X$} as the least infinite cardinal $\kappa$ such that every open $\I$-cover of $X$ has an $\I$-subcover of cardinality $\leq\kappa$ (see \cite{mcoy}).
Theorem 4.7.1 of \cite{mcoy} states that $t\big( C_\I(X)\big)=\I\mbox{-}\Lin(X)$. In \cite{pyt}, Pytkeev proved that if $Y$ is a metrizable space, then  $t(C_p(X,Y))=\aleph_0$ if and only if the Lindel\"{o}f number $l(X^n)$ of $X^n$ is countable for every $n\geq 1$. The next theorem extends these results and has a similar proof.

\begin{theorem} \label{t:Cp-tight}
Let $Y$ be a metric space containing at least two points, and let $X$ be a $Y_\I$-Tychonoff space for some ideal $\I$ of compact sets in $X$. If $H$ is a relatively $Y_\I$-Tychonoff subspace of $C_\I(X,Y)$, then $t(H)=\I\mbox{-}\Lin(X)$. In particular, $t\big( C_\I(X,Y)\big)=\I\mbox{-}\Lin(X)$.
\end{theorem}

\begin{proof}
%
First we show that $\I\mbox{-}\Lin(X)\leq t(H)$. To this end, fix two distinct points $y,z\in Y$ and let $\U$ be an open $\I$-cover of $X$ (we assume that $X\not\in\U$).  Then for each $F\in\I$, there is a $U_F\in\U$ such that $F\subseteq U_F$. Since $X$ is $Y_\I$-Tychonoff and $H$ is relatively $Y_\I$-Tychonoff, for each $F\in\I$, choose an $f_F\in H$ such that $f_F(F)=\{ y\}$ and $f_F(X\SM U_F)=\{z\}$. In particular, the constant function $\yyy$ belongs to the closure of the family $A=\{ f_F: F\in\I\}$, and $\yyy\not\in A$ because $U_F\not= X$. Then there exists a subfamily $B$ of $A$ with cardinality $\leq  t(H)$, which contains $\yyy$ in its closure. Define $\V:=\{ U_F: f_F\in B\}$. To see that $\V$ is an $\I$-subcover of $\U$, let $F\in\I$ and fix an open neighborhood $V$ of $y$ such that $z\not\in V$. Set $W:=W[\yyy;F,V]$. Then $W$ is a neighborhood of $\yyy$, and therefore there is an $f_G\in B\cap W$ for some $G\in \V$. Now for any $x\in F$, we have $f_G(x)\in V$; while $f_G(x)=z$ for any $x\in X\SM U_G$. Therefore $F\subseteq U_G$, so that $\V$ is an $\I$-subcover of $\U$ and has cardinality $\leq t(H)$. Thus $\I\mbox{-}\Lin(X)\leq t(H)$.

Let us prove the reverse inequality $t(H)\leq\I\mbox{-}\Lin(X)$. Since $Y$ embeds into a Banach space $E$, it follows that $H$ is a subspace of $C_\I(X,E)$. Therefore it suffices to prove that $t\big( C_\I(X,Y)\big)\leq \I\mbox{-}\Lin(X)$, where $Y$ is a Banach space. As the space $C_\I(X,Y)$ is homogeneous, it is sufficient to show that the tightness of $C_\I(X,Y)$ at zero function $\mathbf{0}$ is less than or equal to $\I\mbox{-}\Lin(X)$. To this end, let $A$ be a subset of $C_\I(X,Y)$ such that $\mathbf{0}\in  \overline{A}\SM A$. For every $\e>0$, denote by $D_\e$ the open ball of radius $\e$ centered at $0\in Y$. For each $F\in\I$ and a natural number $n\geq 1$, choose a function $h_{n,F}\in A\cap [F;D_{1/n}]$, and set
\[
W(n,F):= \{ x\in X: h_{n,F}(x) \in D_{1/n}\}.
\]
Then for each $n\geq 1$, the family $\mathcal{W}_n :=\{ W(n,F): F\in \I\}$ is an open $\I$-cover of $X$. So each $\mathcal{W}_n$ has an $\I$-subcover $\mathcal{V}_n$ of cardinality $\leq\I\mbox{-}\Lin(X)$. Define
\[
A' := \{ h_{n,F}: n\geq 1 \; \mbox{ and }\; W(n,F)\in \mathcal{V}_n\}
\]
Clearly, $A'\subseteq A$ and $|A'|\leq\I\mbox{-}\Lin(X)$, and $\mathbf{0}$ is in the closure of $A'$ in $C_\I(X,Y)$. Thus $t\big( C_\I(X,Y)\big)\leq \I\mbox{-}\Lin(X)$ as desired. \qed
\end{proof}

\begin{corollary} \label{c:Cp-tight}
Let $X$ be a Tychonoff space, $\I$ be an ideal of compact subsets of $X$, $Y$ be  a metrizable locally convex space, and let $E=C_\I(X,Y)$, $C^b_\I(X,Y)$ or $C^{rc}_\I(X,Y)$. Then $t(E)=\I\mbox{-}\Lin(X)$.
\end{corollary}

\begin{proof}
By Proposition \ref{p:Y-good=>good-I}, the space $X$ is $Y_\I$-Tychonoff. By Proposition \ref{p:Y-LCS=>good-I}, the space $E$ is a  relatively $Y_\I$-Tychonoff subspaces of $C_\I(X,E)$. Now Theorem \ref{t:Cp-tight} applies. \qed
\end{proof}

Below we shall use the well-known fact that the Tychonoff product $\w^{\w_1}$ of $\w_1$ discrete spaces $\w$ is not normal (for a more general assertion see \cite[Theorem~3.10]{Gabr-B1}).
\begin{proposition} \label{p:normality-func}
Let $Y$ be a non-pseudocompact Tychonoff space, $X$ be a $Y$-Tychonoff space containing a discrete family $\U=\{ U_i\}_{i\in\w_1}$ of open subsets,  and  let $\I$ be an ideal of compact subsets of $X$. Then each subspace $H$ of $(Y^X,\tau_\I)$ consisting of $C_\I(X,Y)$ is not normal.
\end{proposition}

\begin{proof}
As $Y$ is not pseudocompact, it contains an infinite, closed and discrete  subspace $D=\{ y_n\}_{n\in\w}$. Since $X$ is Tychonoff by Proposition \ref{p:Y-good=>good}, for every $i\in \w_1$, choose a point $x_i\in U_i$ and an open neighborhood $V_i$ of $x_i$ such that $\overline{V_i}\subseteq U_i$. For every $i\in \w_1$ and $n\in\w$, by $Y$-Tychonoffness of $X$, there is a continuous function $f_{n,i}\in C(X,Y)$ such that
\[
f_{n,i}(x_i)=y_n \;\; \mbox{ and } \;\; f_{n,i}(X\SM V_i)=\{ y_0\}.
\]
For every function $\chi:\w_1 \to \w$, define a function $F_\chi: X\to Y$ by
\[
F_\chi (x):= \left\{
\begin{aligned}
f_{\chi(i),i}(x), & \mbox{ if } x\in V_i \mbox{ for some } i\in\w_1,\\
y_0, &  \mbox{ if } x\in X\SM \bigcup_{i\in\w_1} V_i.
\end{aligned} \right.
\]
Since the family $\U$ is discrete, it is easy to see that the function $F_\chi $ is continuous.

We claim that the map $p:\w^{\w_1} \to H$ defined by $p(\chi):= F_\chi $ is a homeomorphism onto its image $p\big(\w^{\w_1}\big)\subseteq H$. Indeed, it is clear that $p$ is a  bijection. To prove that $p$ is continuous, fix $\chi\in \w^{\w_1}$ and a standard neighborhood $W[p(\chi); \F,\V]$ of $p(\chi)$ in $(Y^X,\tau_\I)$, where $\F$ is a finite subfamily of $\I$ and $\V$ is a corresponding finite family of open sets in $Y$. Since the family $\U$ is discrete, the set $J$ of all $i\in\w_1$ for which the set $K\cap U_i$ is not empty for some $K\in\F$ is finite. For every $j\in J$, set $W_j:=\{ \chi(j)\}$ and let $\mathcal{W}=\{ W_j:j\in J\}$. Then $p\big( W[\chi; J,\mathcal{W}]\big) \subseteq W[p(\chi); \F,\V]$. Thus $p$ is continuous.

To show that $p$ is also open, let $\chi\in \w^{\w_1}$ and let
\[
W:=\{ \eta\in \w^{\w_1}: \eta(i)=\chi(i) \mbox{ if } i\in F \mbox{ for some finite } F\subseteq \w_1 \}
\]
be a standard neighborhood of $\chi$. For every $i\in F$, choose an open neighborhood $\mathcal{O}_i$ of $F_\chi(x_i)$ such that $\mathcal{O}_i\cap D=\{ F_\chi(x_i)\}$. Then $p(W)$ contains the standard neighborhood
\[
\bigcap_{i\in F} W\big[ F_\chi; \{ x_i\},  \{ \mathcal{O}_i\}\big] \cap p\big(\w^{\w_1}\big)
\]
of $F_\chi$. Thus $p$ is open.

Since the space $\w^{\w_1}$ is not normal, to show that also $H$ is not normal it suffices to prove that $p\big(\w^{\w_1}\big)$ is closed in $H$. Fix an arbitrary $h\in \overline{p\big(\w^{\w_1}\big)}$. It is clear that $h(x)=y_0$ for every $x\in X\SM \bigcup_{i\in\w_1} V_i$. Fix an $i\in\w_1$. Since $D$ is closed and discrete in $Y$ it follows that there is an $n(i)\in\w$ such that $h(x_i)=f_{n(i),i}(x_i)=y_{n(i)}\in D$. Take an open neighborhood $O_i$ of $y_{n(i)}$ such that $O_i\cap D=\{y_{n(i)}\}$. Then for every $x\in V_i$ and each neighborhood $O_x$ of $x$, the standard neighborhood
$
W\big[h;\{x_i\}, \{O_i\}\big]\cap W\big[h;\{x\}, \{O_x\}\big] \cap H
$ of $h$ contains only those functions  $f\in p\big(\w^{\w_1}\big)$ for which  $f(x)=f_{n(i),i}(x)$. Therefore $h{\restriction}_{V_i}=f_{n(i),i}{\restriction}_{V_i}$. Hence $h=F_\chi$, where $\chi\in \w^{\w_1}$ is defined by $\chi(i)=n(i)$ for all $i\in\w_1$. Thus $p\big(\w^{\w_1}\big)$ is closed in $H$. \qed
\end{proof}

\begin{remark} {\em
In Proposition \ref{p:normality-func} the assumption on $Y$ to be non-pseudocompact is essential. Indeed, let $Y=[0,1]$ and $X=\bigcup_{i\in\w_1} S_i$ be the topological sum of countable compact spaces $S_i$. It follows from Proposition \ref{p:B1-countable} below that $B_1(X,Y)=\prod_{i\in\w_1} Y^{S_i}=Y^X$, and hence $B_1(X,Y)$ is a compact space. However, if we consider  the case $H=C_\I(X,Y)$, then the space $H$ is still not normal because all spaces $C_\I(S_i,Y)$ contain infinite, closed and discrete subspaces, see Lemma 1 of \cite{Pol-1974}.\qed}
\end{remark}

To prove the next proposition we need the following mild extension property.
\begin{definition} \label{def:extention} {\em
Let $X$ and $Y$ be topological spaces, and let $\I$ be an ideal of compact  subsets in $X$. We say that $X$ has  the {\em $Y_\I$-extension property} if for every $F\in\I$ and each $f\in C(F,Y)$ there is an ${\bar f} \in C(X,Y)$ such that ${\bar f}{\restriction}_F =f$.\qed}
\end{definition}


Let $X$ be a Tychonoff space. As usual we denote by $w(X)$ and $d(X)$ the {\em weight of $X$} and the {\em density of $X$}, respectively.

\begin{proposition} \label{p:C(X,Y)-net-weight}
Let $Y$ be a Hausdorff topological space containing at least two points, and let  $X$ be a $Y_{\I_X}$-Tychonoff space for some ideal $\I_X$ of compact subsets in $X$. Let $S$ be  a  $Y_{\I_S}$-Tychonoff space for some ideal $\I_S$ of compact subsets in $S$ such that
\begin{enumerate}
\item[{\rm (a)}] there is a continuous bijection $\phi$ from $X$ onto $S$ such that $\phi(\I_X)\subseteq \I_S$;
\item[{\rm (b)}] $S$ has the $Y_{\phi(\I_X)}$-extension property.
\end{enumerate}
Then $d\big(C_{\I_X}(X,Y)\big)\leq d\big(C_{\I_S}(S,Y)\big)$.
\end{proposition}

\begin{proof}
It is sufficient to show that if $D$ is a dense subset of $C_{\I_S}(S,Y)$, then $\phi^\ast(D)$ is dense in $C_{\I_X}(X,Y)$, where $\phi^\ast: C_{\I_S}(S,Y)\to C_{\I_X}(X,Y), \phi^\ast(f):=f\circ \phi,$ is the adjoint map of $\phi$. Fix an $f\in C_{\I_X}(X,Y)$ and its standard neighborhood $W[f; \F,\U]$, where $\F\subseteq \I_X$ is finite. Then the bijectivity of $\phi$ and (b) imply that there is an $h\in C_{\I_S}(S,Y)$ such that $h{\restriction}_{\bigcup\F} =f{\restriction}_{\bigcup\F}$. Take $g\in D$ such that $g\in W[h; \F,\U]\subseteq C_{\I_S}(S,Y)$. It is clear that $\phi^\ast(g)=g\circ\phi \in W[f; \F,\U]$. Thus $\phi^\ast(D)$ is dense in $C_{\I_X}(X,Y)$ and hence $d\big(C_{\I_X}(X,Y)\big)\leq d\big(C_{\I_S}(S,Y)\big)$.\qed
\end{proof}

The {\em weak weight $ww(X)$} of a Tychonoff space $X$ is the least infinite cardinality of $w(Z)$ of a Tychonoff space $Z$ such that there is a continuous bijection from $X$ onto $Z$, see \cite{mcoy}. Analogously we define
\begin{definition} \label{def:weak-weight} {\em
Let $X$ and $Y$ be topological spaces, and let $\I$ be an ideal of compact subsets in $X$ such that $X$ has the $Y_\I$-extension property.
The {\em $Y_\I$-weak weight of $X$} (for short $Y_\I\mbox{-}ww(X)$)  is the least infinite cardinality of $w(Z)$ of a $Y$-Tychonoff space $Z$  such that there is a continuous bijection $\phi$ from $X$ onto $Z$ and $Z$ has the $Y_{\phi(\I)}$-extension property.\qed}
\end{definition}
For the sake of simplicity we set $Y_p\mbox{-}ww(X):= Y_{\F(X)}\mbox{-}ww(X)$ and $Y_k\mbox{-}ww(X):= Y_{\KK(X)}\mbox{-}ww(X)$.

\begin{lemma} \label{l:extention}
\begin{enumerate}
\item[{\rm (i)}] Every $Y$-Tychonoff space $X$ has the  $Y_{\F(X)}$-extension property.
\item[{\rm (ii)}] 
 Each Tychonoff space $X$ has the  $\IR_{\I}$-extension property for every ideal $\I$ of compact subsets of $X$.
\item[{\rm (iii)}] Let $\I$ be an ideal of compact subsets of a Tychonoff space $X$. Then $\IR_\I\mbox{-}ww(X)=ww(X)$.
\item[{\rm (iv)}] $ww(X)\leq Y_\I\mbox{-}ww(X)$ for every Tychonoff space $Y$ and each ideal $\I$ of compact subsets of a $Y$-Tychonoff space $X$ with the $Y_\I$-extension property.
\item[{\rm (v)}] If $X$ is a $Y$-Tychonoff (for example, $Y$-$z$-Tychonoff) compact space, then $Y_p\mbox{-}ww(X)=w(X)$.
\end{enumerate}
\end{lemma}

\begin{proof}
(i) follows from definitions, (ii) follows from  \cite[3.11(c)]{GiJ}, and (iii) follows from (ii). Finally, (iv) follows from (ii) and the fact that every $Y$-Tychonoff space is Tychonoff, see Proposition \ref{p:Y-good=>good}.

(v) Recall that, by Proposition \ref{p:Y-z-Tychonoff-exa},  every $Y$-$z$-Tychonoff space is $Y$-Tychonoff. By (i), $X$ has the  $Y_{\F(X)}$-extension property. Since every continuous bijection of a compact space is a homeomorphism, by Definition \ref{def:weak-weight}, we have $Y_p\mbox{-}ww(X)=w(X)$. \qed
\end{proof}

Taking into account Lemma \ref{l:extention}, for $Y=\IR$, the next assertion is proved in Theorem 4.2.1 of \cite{mcoy}.

\begin{proposition} \label{p:C(X,Y)-net-weight-1}
Let $Y$ be a Hausdorff topological space containing at least two points, and let  $X$ be a $Y_{\I_X}$-Tychonoff space for some ideal $\I_X$ of compact subsets in $X$. If $X$ has the $Y_{\I_X}$-extension property, then
\begin{enumerate}
\item[{\rm (i)}] $d(Y)\leq  d\big(C_{\I_X}(X,Y)\big)\leq Y_\I\mbox{-}ww(X)\cdot w(Y)$;
\item[{\rm (ii)}] $ww(X)\leq w(Y)\cdot d\big(C_{\I_X}(X,Y)\big)$.
\end{enumerate}
\end{proposition}

\begin{proof}
(i) First we show that $d\big(C_{\I_X}(X,Y)\big)\leq Y_\I\mbox{-}ww(X)\cdot w(Y)$. To this end, let $\phi:X\to Z$ be a continuous bijection, where $Z$ is a $Y$-Tychonoff space with the $Y_{\phi(\I_X)}$-extension property and such that $w(Z)=Y_\I\mbox{-}ww(X)$. For simplicity, set  $\I_Z:=\phi(\I_X)$ (so $\I_Z$ is an ideal of compact sets in $Z$). By Proposition \ref{p:C(X,Y)-net-weight}, it suffices to show that $d\big(C_{\I_Z}(Z,Y)\big)\leq w(Z)\cdot w(Y)$. Choose an open base $\mathcal{B}_Z$ for $Z$ which is closed under taking finite unions and such that $|\mathcal{B}_Z|=w(Z)$, and let $\mathcal{B}_Y$ be an open base of $Y$ such that $|\mathcal{B}_Y|=w(Y)$.  For every finite subfamilies $\U=\{ U_1,\dots,U_n\}\subseteq \mathcal{B}_Z$ and $\V=\{ V_1,\dots,V_n\}\subseteq \mathcal{B}_Y$, choose (if this is possible) a function $f_{\U,\V}\in C(Z,Y)$ such that $f_{\U,\V}(U_i)\subseteq V_i$ for every $i=1,\dots,n$ (i.e., $f_{\U,\V}\in [\U;\V]$). Set
\[
D:=\{ f_{\U,\V}: \U\subseteq \mathcal{B}_Z \mbox{ and } \V\subseteq \mathcal{B}_Y \mbox{ are finite}\}.
\]
By construction, $|D|\leq w(Z)\cdot w(Y)$. We show that $D$ is dense in $C_{\I_Z}(Z,Y)$. Indeed, let $f\in C_{\I_Z}(Z,Y)$ and let $W[f;\F,\V]$ be a basic open neighborhood of $f$ in $C_{\I_Z}(Z,Y)$, where $\F=\{ F_1,\dots,F_n\} \subseteq \I(Z)$ and $\V=\{ V_1,\dots,V_n\}$ is a finite family of open sets in $Y$. Since $f$ is continuous and $\mathcal{B}_Z$ is closed under taking finite unions, for every $i=1,\dots,n$, there is an open $U_i\in \mathcal{B}_Z$ such that $F_i\subseteq U_i$ and $f(U_i)\subseteq V_i$. Therefore, the set $ [\U;\V]$ is not empty. It is clear that $f_{\U,\V}\in [\U;\V]\subseteq W[f;\F,\V]$, and hence $D$ is dense in $C_{\I_Z}(Z,Y)$. Thus $d\big(C_{\I_Z}(Z,Y)\big)\leq w(Z)\cdot w(Y)$.

To prove the inequality $d(Y) \leq d\big(C_{\I_X}(X,Y)\big)$, fix a dense subset $\mathcal{R}$ of $C_{\I_X}(X,Y)$ such that $|\mathcal{R}|=d\big(C_{\I_X}(X,Y)\big)$.
Take an arbitrary $x_0\in X$ and put $D_Y :=\{ f(x_0): f\in \mathcal{R}\}$. Then $D_Y$ is dense in $Y$ since, otherwise, there would be an open subset $U$ of $Y$ such that $U\cap D_Y=\emptyset$. But then the open subset $[\{x_0\};U]$ of $C_{\I_X}(X,Y)$ is nonempty (since it contains the constant function $\yyy$ for every $y\in U$) and does not contain elements of the dense family $\mathcal{R}$, a contradiction. Thus $d(Y)\leq |D_Y|\leq |\mathcal{R}|=d\big(C_{\I_X}(X,Y)\big)$.
\smallskip

(ii) As above, let $\mathcal{R}$ be a dense subset of $C_{\I_X}(X,Y)$ of cardinality $d\big(C_{\I_X}(X,Y)\big)$. Define a map $T:X\to Y^{\mathcal{R}}$ by $T(x):=\big( f(x)\big)_{f\in \mathcal{R}}$, where $Y^{\mathcal{R}}$ is endowed with the product topology. Clearly, $T$ is continuous. Since $X$ is $Y_{\I_X}$-Tychonoff and $Y$ is Hausdorff, the family  $\mathcal{R}$ separates the points of $X$, and hence $T$ is also one-to-one. Therefore,
\[
ww(X)\leq w\big(T(X)\big) \leq w(Y)\cdot |\mathcal{R}|= w(Y)\cdot d\big(C_{\I_X}(X,Y)\big). \mbox{\qed}
\]
\end{proof}

\begin{corollary} \label{c:density-C(X,Y)}
Let $Y$ be a separable metric space containing at least two points and let $X$ be  $Y$-Tychonoff space. Then $ww(X)\leq d\big(C_{p}(X,Y)\big)\leq Y_p\mbox{-}ww(X)$.  In particular, if $X$ is compact, then $d\big(C_{p}(X,Y)\big)=w(X)$.
\end{corollary}

\begin{proof}
The first inequality $ww(X)\leq d\big(C_{p}(X,Y)\big)$ follows from (ii) of Proposition \ref{p:C(X,Y)-net-weight-1}.
The second inequality $d\big(C_{p}(X,Y)\big)\leq Y_p\mbox{-}ww(X)$ follows from Lemma \ref{l:extention}(i), Proposition \ref{p:C(X,Y)-net-weight-1}(i) and the fact $w(Y)=\aleph_0$.
Finally, if additionally $X$ is compact, the last assertion follows from (v) of Lemma \ref{l:extention} and the trivial fact that $ww(X)=w(X)$.
\qed
\end{proof}

Lemma \ref{l:extention}(iii) and Proposition \ref{p:C(X,Y)-net-weight-1} immediately imply the next result which is also proved in Theorem 4.2.1 of \cite{mcoy}.
\begin{corollary} \label{c:density-Cp}
If $X$ is a Tychonoff space and $\I$ is an ideal of compact subsets of $X$, then $d\big(C_{\I}(X)\big)= ww(X)$.
\end{corollary}


\section{Topological properties of spaces of Baire functions} \label{sec:Baire-topology}


Recall that the {\em cellularity $c(X)$}  of a topological space $X$ is the minimal infinite cardinal $\kappa$ such that every disjoint family of  open sets in $X$ has cardinality less than or equal to $\kappa$. Denote by $\psi(x,X)$ and $\psi(X)$ the pseudocharacter of $X$ at a point $x\in X$ and the pseudocharacter of $X$, respectively. We shall use the following results.

\begin{proposition} \label{p:cell-P-space}
Let $Z$ be  a Tychonoff $P$-space. If $\psi(Z)>\aleph_0$, then $c(Z)>\aleph_0$. In particular, if $X$ is a Tychonoff space such that $\psi(X_{\aleph_0})>\aleph_0$, then $c(X_{\aleph_0})>\aleph_0$.
\end{proposition}

\begin{proof}
Fix a point $z\in Z$ in which the pseudocharacter $\psi(z,Z)$ of $Z$ is uncountable. We construct an uncountable family of pairwise disjoint open sets in $Z$ by transfinite induction. For $i=0$, set $V_0:=Z$. Choose a clopen neighborhood $V_1$ of $z$ such that $V_0\SM V_1 \not=\emptyset$ and set $U_1:=V_0\SM V_1$. Assume that for every countable ordinal $\alpha$ and each $i<\alpha$ we have constructed a family $\{ V_i\}_{i<\alpha}$ of clopen neighborhoods of $z$ and a disjoint family $\{ U_i\}_{i<\alpha}$ of clopen subsets of $Z$ such that
\[
U_i\subseteq \bigcap_{j<i} V_j \SM V_i
\]
for every $i<\alpha$. Since $Z$ is a $P$-space, the set  $W:=\bigcap_{j<\alpha} V_j$ is an open neighborhood of $z$. As $\psi(z,Z)>\aleph_0$, there is a clopen neighborhood $V_\alpha \subseteq W$ of $z$ such that $W\SM V_\alpha\not=\emptyset$. Choose a clopen nonempty subset $U_\alpha$ of $W\SM V_\alpha$. It is clear that the uncountable family $\{ U_i: i<\w_1\}$ is  disjoint. Thus $c(Z)>\aleph_0$.

The last assertion follows from the fact that $X_{\aleph_0}$ is a $P$-space. \qed
\end{proof}

\begin{corollary} \label{c:cell-countable}
Let $X$ be a Tychonoff space. Then:
\begin{enumerate}
\item[{\rm (i)}] If $\psi(X_{\aleph_0})=\aleph_0$ (for example $\psi(X)=\aleph_0$),  then $X_{\aleph_0}$ is discrete.
\item[{\rm (ii)}] $c(X_{\aleph_0})=\aleph_0$ if and only if $X$ is countable.
\end{enumerate}
\end{corollary}

\begin{proof}
(i) immediately follows from the fact that  $X_{\aleph_0}$ is a $P$-space.

(ii) Assume that $X_{\aleph_0}$ has countable cellularity. Then, by Proposition \ref{p:cell-P-space},  $\psi(X_{\aleph_0})=\aleph_0$ and hence, by (i), $X_{\aleph_0}$ is discrete. But since $c(X_{\aleph_0})=\aleph_0$ it follows that $X$ is countable.

Conversely, if $X$ is countable, then the equality $c(X_{\aleph_0})=\aleph_0$ holds trivially. \qed
\end{proof}

Let $X$ and $Y$ be topological spaces.  Define $B(X,Y):= \bigcup_{\alpha\in\w_1} B_\alpha(X,Y)$. A function $f:X\to Y$ is called {\em Baire} if $f\in B(X,Y)$.  If $Y=\IR$, set $B(X):=B(X,\IR)$.
If $X$ is Tychonoff, it is easy to see that every function $f\in B(X)$ is continuous in the Baire topology. We shall use the following easy assertion which allows us to reduce the study of topological properties of the corresponding spaces of Baire functions to topological properties of spaces of continuous functions.

\begin{proposition} \label{p:B(X,Y)-C(X,Y)}
Let $Y$ be a topological space, and let $X$ be a Tychonoff space. Then the adjoint map $i^\ast$ of the identity mapping $i: X_{\aleph_0}\to X$ is an embedding of $B(X,Y)$ into $C_p(X_{\aleph_0},Y)$.
\end{proposition}

\begin{proof}
Clearly, the adjoint map $i^\ast: Y^X\to Y^{X_{\aleph_0}}, i^\ast(f)=f\circ i,$  is a topological isomorphism and $i^\ast\big( C_p(X,Y)\big) \subseteq C_p(X_{\aleph_0},Y)$. Since $X$ is Tychonoff, the space $X_{\aleph_0}$ is a $P$-space. Therefore the pointwise limit of a sequence of continuous functions in $C_p(X_{\aleph_0},Y)$ is continuous. Thus, by the definition of $B(X,Y)$, the map  $i^\ast$ is an embedding.\qed
\end{proof}

Besides the classes $B_\alpha(X,Y)$  we shall consider also the  {\em stable} Baire-$\alpha$ classes $B_\alpha^{st}(X,Y)$ of functions, which were introduced and studied by Cs\'asz\'ar and  Laczkovich \cite{CL1}. Set $B_0^{st}(X,Y):=C_p(X,Y)$.  We say that a sequence $\{f_n\}_{n\in\w}\subseteq Y^X$ {\em stably converges} to a function $f\in Y^X$ if for every $x\in X$ the set $\{n\in\w:f_n(x)\not= f(x)\}$ is finite. Then the class $B_1^{st}(X,Y)$  is defined as the family of all functions from $Y^X$ which are limits of stably convergent sequences of continuous functions. For every countable ordinal $\alpha>1$,  denote by $B_\alpha^{st}(X,Y)\subseteq Y^X$ the family of all functions $f:X\to Y$ which are pointwise limits of function sequences from $\bigcup_{\beta<\alpha}B_\beta^{st}(X,Y)$. Functions in the family $B_\alpha^{st}(X,Y)$ are called the {\em functions of stable Baire-$\alpha$ class}. It is clear that
\[
C_p(X,Y)\subseteq   B_\alpha^{st}(X,Y) \subseteq B_\alpha(X,Y) \subseteq B(X,Y)
\]
for every countable ordinal $\alpha$. Note that the inclusion $B_1^{st}(X,Y)\subseteq B_1(X,Y)$ can be strict.   If $Y=\IR$, set $B_1^{st}(X):=B_1^{st}(X,\IR)$.  For a function $f:X\to Y$, we set
\[
\sigma(f):= \big\{ h\in Y^X: \mbox{ the set } \{ x\in X: h(x)\not= f(x)\} \mbox{ is finite}\big\}.
\]

It is very useful to know concrete constructions of Baire one functions.
The following lemma extends Lemmas 3.3 and 3.4 of \cite{Gabr-B1}.

\begin{lemma} \label{l:Baire-1}
Let $Y$ be a metric space containing at least two points, and let $X$ be an infinite $Y$-Tychonoff space of countable pseudocharacter. Let $\{ U_n\}_{n< N}$, $0< N\leq\infty$, be a disjoint family of open subsets of $X$ and let $x_n\in U_n$ for every $n<N$. Then, for every $g_0,g_1,\dots, g_N \in Y$, the function
\[
f(x):= \left\{
\begin{aligned}
g_n, & \mbox{ if } x=x_n \mbox{ for some } n<N, \\
g_N, & \mbox{ if } x\in X\setminus \{ x_n: n<N\}
\end{aligned} \right.
\]
belongs to $B_1^{st}(X,Y)$.  In particular, $\sigma\big(\mathbf{g}\big) \subseteq B_1^{st}(X,Y)$ for every $g\in Y$.
\end{lemma}

\begin{proof}
We prove the lemma only for the case $N=\w$.
For every $n\in\w$, choose a decreasing sequence $\{ V_{k,n}:k\in\w\}$ of open neighborhoods of $x_n$ such that $V_{0,n} \subseteq U_n$ and $\bigcap_{k\in\w} V_{k,n}=\{ x_n\}$. Since $X$ is $Y$-Tychonoff, for every $k\in\w$, there is a function $f_k\in C(X,Y)$ such that
\[
f_k(x_i)=g_i \; \mbox{ for }\; i=0,\dots,k, \;\mbox{ and }\; f_k\Big( X\SM \bigcup_{i=0}^k V_{k,i} \Big)\subseteq \{g_\w\}.
\]
Now the choice of the open sets $V_{k,n}$ easily implies that the sequence $\{f_k\}_{k\in\w}$ stably converges to $f$. \qed
\end{proof}

\begin{remark} {\em
The condition on $\{ U_n\}_{n\in\w}$ to be disjoint in Lemma \ref{l:Baire-1} is essential. Indeed, consider $X=Y=\IR$, $g_i=1$ for all finite $i$ and $g_\w=0$, and let $\{x_n:n\in\w\}$ be an enumeration of the rational numbers. Then the function $f$ (= the Dirichlet function) does not belong to $B_1(X)$ because it does not have points of continuity.\qed}
\end{remark}

\begin{lemma} \label{l:G-del-z-modif}
Let $Y$ be a normal space containing at least two points, and let $X$ be a $Y$-$z$-Tychonoff space.
Then for every  disjoint zero-sets $F_1,\dots,F_n$ in $X$ and each points $y_0,\dots,y_n\in Y$ there exists a function $f\in B_1^{st}(X,Y)$ such that $f\big(X\SM \bigcup_{i=1}^n F_i\big)\subseteq \{ y_0\}$ and $f(F_i)=\{ y_i\}$ for every $i=1,\dots,n$.
\end{lemma}

\begin{proof}
Since the set $F:=\bigcup_{i=1}^n F_i$ is a zero-set in $X$, there is a $t\in C(X)$ such that $F=t^{-1}(0)$. Therefore the set $X\SM F = \bigcup_{n\in\w} S_n$, where $S_n=\{ x\in X: |t(x)|\geq \tfrac{1}{n+1}\}$, is the union of the increasing sequence $\{ S_n \}_{n\in\w}$ of zero-sets in $X$. Since $S_n\cap F=\emptyset$, the $Y$-$z$-Tychonoffness of $X$ implies that there is an $f_n\in C(X,Y)$ such that $f_n(S_n)\subseteq \{ y_0\}$ and $f_n(F_i)=\{ y_i\}$ for every $i=1,\dots,n$. It is clear that $f_n$ stably converges to the function $f$.\qed
\end{proof}

The following lemma plays a crucial role in the proofs of the main results of this section.
\begin{lemma} \label{l:G-del-modif}
Let $Y$ be a Tychonoff space of  countable pseudocharacter containing at least two points, and let $X$ be a $Y$-Tychonoff space. Assume that $A=\{a_1,\dots,a_n\}$ is a subset of a zero-set $F$ in $X$, $h:A\to Y$ is a function and $y_0\in Y$. Let $h(A):=\{ y_1,\dots,y_k\}$ with distinct $y_1,\dots,y_k\in Y$. Then there exist disjoint zero-sets $Z_1,\dots,Z_k$ in $X$ and a function $f\in B_1^{st}(X,Y)$ such that
\begin{enumerate}
\item[{\rm (i)}] for every $j=1,\dots,n$, there is an $i_j\in\{1,\dots,k\}$ such that $a_j\in Z_{i_j}$;
\item[{\rm (ii)}] $\bigcup_{i=1}^k Z_i \subseteq F$ and $f(Z_i)=\{y_i\}$ for every $i=1,\dots,k$;
\item[{\rm (iii)}] $f\big( X\SM\bigcup_{i=1}^k Z_i\big) \subseteq \{ y_0\}$ and $f{\restriction}_A=h$.
\end{enumerate}
Moreover, if $A=\{a\}$, then there is a sequence $\{ f_n\}_{n\in\w}\subseteq C(X,Y)$ such that $f_n(a)=h(a)$ for all $n\in\w$ and $\bigcap_{n\in\w} f_n^{-1}\big( h(a)\big) \subseteq F$.
\end{lemma}

\begin{proof}
We shall use the following fact: a Tychonoff space has countable pseudocharacter if and only if each its point is a zero-set. So $\{ y\}$ is a zero-set in $Y$ for each $y\in Y$.

Set $F_0:= F$. Then, as in the proof of Lemma \ref{l:G-del-z-modif}, $X\SM F_0 = \bigcup_{n\in\w} S_{n,0}$, where $\{ S_{n,0} \}_{n\in\w}$ is an increasing sequence  of zero-sets in $X$. Since $X$ is $Y$-Tychonoff, there is a function $f_0\in C(X,Y)$ such that $f_0{\restriction}_A = h$ and $f_0(S_{0,0})\subseteq \{y_0\}$.

For every $i=1,\dots,k$, set $F_0^i := F_0 \cap f_0^{-1}\big( y_i\big)$. Since $\{ y_i\}$ is a zero-set in $Y$ and $y_1,\dots,y_k$ are distinct, it follows that the sets $F_0^1,\dots,F_0^k$ are  disjoint zero-sets in $X$, and hence the set $F_1 :=\bigcup_{i=1}^k F_0^i $ is also a zero-set in $X$. Clearly, $A\subseteq F_1\subseteq F$ and $f_0(F_0^i)=\{ y_i\}$ for every $i=1,\dots,k$. Then $X\SM F_1 = \bigcup_{n\in\w} S_{n,1}$, where $\{ S_{n,1} \}_{n\in\w}$ is an increasing sequence  of zero-sets in $X$. Since $X$ is $Y$-Tychonoff, there is a function $f_1\in C(X,Y)$ such that $f_1{\restriction}_A = h$ and $f_1(S_{1,0}\cup S_{0,1})\subseteq \{y_0\}$.

For every $i=1,\dots,k$, set $F_1^i := F_0^i \cap f_1^{-1}\big( y_i\big)$. Since $\{ y_i\}$ is a zero-set in $Y$ and $y_1,\dots,y_k$ are distinct, it follows that the sets $F_1^1,\dots,F_1^k$ are disjoint zero-sets in $X$, and hence the set $F_2 :=\bigcup_{i=1}^k F_1^i $ is also a zero-set in $X$. Clearly, $A\subseteq F_2\subseteq F_1$ and $f_1(F_1^i)=\{ y_i\}$ for every $i=1,\dots,k$.  Then $X\SM F_2 = \bigcup_{n\in\w} S_{n,2}$, where $\{ S_{n,2} \}_{n\in\w}$ is an increasing sequence  of zero-sets in $X$. Since $X$ is $Y$-Tychonoff, there is a function $f_2\in C(X,Y)$ such that $f_2{\restriction}_A = h$ and
\[
f_2(S_{2,0}\cup S_{1,1} \cup S_{0,2})\subseteq \{y_0\}.
\]

Proceeding by induction,  for every $t\in\w$, we construct  disjoint zero-sets $F_t^1,\dots,F_t^k$ in $X$ and a function $f_t\in C(X,Y)$ such that
\begin{enumerate}
\item[{\rm (a)}] $F_0^i \cap F_0^j=\emptyset$ for all distinct $i,j\in\{1,\dots,k\}$;
\item[{\rm (b)}] $X\SM \bigcup_{i=1}^k F_t^i = \bigcup_{n\in\w} S_{n,t}$, where $\{ S_{n,t} \}_{n\in\w}$ is an increasing sequence  of zero-sets in $X$;
\item[{\rm (c)}] $ F_{t+1}^i =  F_t^i \cap f_{t+1}^{-1}\big( y_i\big)$ for every $i=1,\dots,k$;
\item[{\rm (d)}]  $f_t{\restriction}_A = h$ and $A\subseteq \bigcup_{i=1}^k F_t^i \subseteq F$;
\item[{\rm (e)}] $f_t\big(S_{t,0}\cup S_{t-1,1} \cup \cdots \cup S_{0,t})\subseteq \{y_0\}$.
\end{enumerate}

For every $i=1,\dots,k$, set $Z_i:= \bigcap_{t\in\w} F_t^i$. It is clear that $Z_i$ is a zero-set in $X$ and, by (a) and (c), the sets $Z_1,\dots,Z_k$ are disjoint. By (d), for every $j=1,\dots,n$, there is an $i_j\in\{1,\dots,k\}$ such that $a_j\in Z_{i_j}$. This proves (i).
Also, by (d),  $\bigcup_{i=1}^k Z_i \subseteq F$.

If $x\in Z_i$ for some $i=1,\dots,k$, then (c) implies that $f_t(x)=y_i$ for all $t\geq 1$. If $x\in X\SM \bigcup_{i=1}^k Z_i$, (c) implies that there is a $t_x\in\w$ such that $x\in X\SM \bigcup_{i=1}^k F_{t_x}^i$ and hence, by (b), there is an $n_x\in\w$ such that $x\in  S_{n_x,t_x}$. Therefore, for every $t>n_x +t_x$, (b) implies $S_{n_x,t_x} \subseteq S_{t-t_x,t_x}$ and hence, by (e),  we have $f_t(x)= y_0$. Thus the sequence $\{ f_t\}$ stably converges to a function $f\in B_1^{st}(X,Y)$ which satisfies (ii) and (iii). The equality $f{\restriction}_A=h$ follows from (d).

To prove the last assertion, assume that $A=\{a\}$. Choose $y_0\in Y$ such that $y_0\not=h(a)$. By construction, $X\SM F = \bigcup_{n\in\w} S_{n,0}$, where $S_{n,0}\subseteq S_{n+1,0}$ for all $n\in\w$. Now, (e) implies $f_n\big(S_{n,0}\big)\subseteq \{y_0\}$, and hence $f_n^{-1}\big( h(a)\big) \subseteq X\SM f_n^{-1}(y_0) \subseteq X\SM S_{n,0}$ for every $n\in\w$. Therefore
\[
\bigcap_{n\in\w} f_n^{-1}\big( h(a)\big) \subseteq \bigcap_{n\in\w} X\SM S_{n,0} =X\SM  \bigcup_{n\in\w}  S_{n,0} =X\SM (X\SM F) =F.
\]
Finally, by (d), $f_n(a)=h(a)$ holds for every $n\in\w$. \qed
\end{proof}


The next proposition gives examples of spaces $X$ for which $\tau_Y= \tau_b$.

\begin{proposition} \label{p:G-del-modif}
Let $Y$ be a perfectly normal space containing at least two points, and let $X$ be a $Y$-Tychonoff space. Then: 
\begin{enumerate}
\item[{\rm (i)}] $\tau_Y= \tau_b$, i.e.  $X_Y= X_{\aleph_0}$.
\item[{\rm (ii)}] $X_{\aleph_0}$ is a $Y$-Tychonoff space.
\end{enumerate}
\end{proposition}

\begin{proof}
(i)  To prove that $\tau_Y\leq \tau_b$ it is sufficient to show that every set of the form $f^{-1}(F)$, where $F$  is a closed set in $Y$ and  $f\in C(X,Y)$, is a zero-set in $X$. Since $Y$ is  perfectly normal, there is a function $h\in C(Y)$ such that $F=h^{-1}(0)$. Setting $g:=h\circ f\in C(X)$ we obtain that $f^{-1}(F)=g^{-1}(0)$ is a zero-set in $X$.

To show the converse inclusion $\tau_Y\geq \tau_b$, fix two distinct points $a,b\in Y$.
Let $F$ be a zero-set of $X$ and let $x$ be any point of $F$. Applying Lemma \ref{l:G-del-modif} to $A=\{x\}$, $y_0=b$ and $h:A\to Y$ defined by $h(x)=a$, we can find a sequence $\{ f_n\}_{n\in\w} \subseteq C(X,Y)$ such that
\[
f_n(x)=a \; (n\in\w) \; \mbox{ and } \; x\in \bigcap_{n\in\w} f^{-1}_n (a)\subseteq F.
\]
As the set $\{ a\}$ is closed in $Y$ and $Y$ is perfectly normal, it follows that $f^{-1}_n (a)\in Z_Y$. Therefore $\bigcap_{n\in\w} f^{-1}_n (a)\in \tau_Y$. Since $x$ was arbitrary we obtain  $F\in \tau_Y$. Thus $\tau_Y\geq \tau_b$.
\smallskip

(ii) Let $A=\{a_1,\dots,a_n\}$ be a finite subset of $X_{\aleph_0}$, $D$ be a closed subset of $X_{\aleph_0}$ such that $D\subseteq X\SM A$, $h:A\to Y$ be a function, and let $y_0\in Y$. Since $X$ is Tychonoff by Proposition \ref{p:Y-good=>good}, there are disjoint zero-sets $S_1,\dots,S_n$ in $X$ such that $a_i\in S_i$ and $S_i\subseteq X\SM D$ for every $i=1,\dots,n$. Setting $F:=\bigcup_{i=1}^n$ and applying Lemma \ref{l:G-del-modif}, one can find a zero-set $S$ in $X$ and a function $f\in B^{st}_1(X,Y)$ such that $A\subseteq S\subseteq \bigcup_{i=1}^n S_i$, $f{\restriction}_A =h$ and $f(X\SM S) \subseteq \{ y_0\}$. Since $B^{st}_1(X,Y)\subseteq  C_p(X_{\aleph_0},Y)$ (see Proposition \ref{p:B(X,Y)-C(X,Y)}), we obtain that $X_{\aleph_0}$ is a $Y$-Tychonoff space. \qed
\end{proof}

Recall some of the most important types of local networks in a topological space. For historical remarks and numerous results related to generalized metric space theory we referee the reader to \cite{gruenhage} or \cite{GK-GMS1}.
A family $\mathcal{N}$ of subsets of a topological space $X$ is
\begin{enumerate}
\item[$\bullet$]   a {\em network at a point} $x\in X$ if for each neighborhood $O_x$ of $x$ there is a set $N\in\mathcal{N}$ such that $x\in N\subseteq O_x$; $\Nn$ is a {\em network} in $X$ if $\mathcal{N}$ is a network at each point $x\in X$.
\item[$\bullet$]  a {\em $cs^\ast$-network at  a point} $x\in X$ if for each sequence $(x_n)_{n\in\NN}$ in $X$ converging to  $x$ and for each neighborhood $O_x$ of $x$ there is a set $N\in\mathcal{N}$ such that $x\in N\subseteq O_x$ and the set $\{n\in\NN :x_n\in N\}$ is infinite; $\Nn$ is a {\em $cs^\ast$-network}  in $X$ if $\mathcal{N}$ is a $cs^\ast$-network at each point $x\in X$.
\item[$\bullet$]  a {\em $cn$-network}  at a point $x\in X$ if for each neighborhood $O_x$ of $x$ the set $\bigcup \{ N \in\Nn : x\in N \subseteq O_x \}$ is a neighborhood of $x$; $\Nn$ is a {\em $cn$-network} in $X$ if $\mathcal{N}$ is a $cn$-network at each point $x\in X$.
\item[$\bullet$]   a {\em $ck$-network}  at  a point $x\in X$ if for any neighborhood $O_x$ of $x$ there is a neighborhood $U_x$ of $x$ such that for each compact subset $K\subseteq U_x$ there exists a finite subfamily $\FF\subseteq\mathcal{N}$ satisfying $x\in \bigcap\FF$ and $K\subseteq\bigcup\FF\subseteq O_x$; $\Nn$ is a {\em $ck$-network}  in $X$ if $\mathcal{N}$ is a $ck$-network at each point $x\in X$.
\item[$\bullet$]  a {\em $cp$-network}  at  a point  $x\in X$ if either $x$ is an isolated point of $X$ and $\{x \} \in\Nn$, or for each subset $A\subset X$ with $x\in \overline{A}\setminus A$ and each neighborhood $O_x$ of $x$ there is a set $N\in\mathcal{N}$ such that $x\in N\subseteq O_x$ and $N\cap A$ is infinite; $\Nn$ is  a  {\em $cp$-network} in $X$ if $\mathcal{N}$ is a $cp$-network at each point $x\in X$.
\end{enumerate}

To unify notations we call a network (at a point $x$ of) $X$ by $0$-network.
\begin{notation} {\em
If $\Nn$ is either a $cp$-, $ck$-,  $cn$-, $cs^\ast$-network or network (at a point $x$) in a topological space $X$, we will say  that $\Nn$ is an {\em $\mathfrak{n}$-network} (at $x$).
Set $\mathfrak{N}=\{ cp, ck,  cn,  cs^\ast, 0\}$.}
\end{notation}

Let $X$ be a topological space.
Recall that a collection $\mathcal{N}$ of subsets of $X$ is said to be {\em locally finite}, if each point in the space has a neighborhood that intersects only finitely many of the sets in $\mathcal{N}$; and $\mathcal{N}$ is called  {\em $\sigma$-locally finite} if it is the union of a countable family of locally finite collections of subsets of $X$. For   $\nn\in\mathfrak{N}$,  the space $X$
\begin{enumerate}
\item[$\bullet$] is an {\em $\nn\mbox{-}\sigma$-space} if $X$ has a $\sigma$-locally finite $\nn$-network;
\item[$\bullet$] has {\em countable $\nn$-character} if $X$ has a countable $\nn$-network at each point $x\in X$.
\end{enumerate}
So $cp\mbox{-}\sigma$-spaces are $\Pp$-spaces, $cs^\ast\mbox{-}\sigma$-spaces are $\aleph$-spaces, and $0\mbox{-}\sigma$-spaces are $\sigma$-spaces, see \cite{GK-GMS1}. Recall also that $X$ is called a {\em cosmic space} (an {\em $\aleph_0$-space} or a {\em $\Pp_0$-space}) if it is regular and has a countable network (resp. a countable $k$- or $cp$-network).

Below we shall use repeatedly the following assertion.

\begin{proposition}[\cite{BG-Baire}] \label{p:B1-countable}
If $X$ is a countable functionally Hausdorff space, then
\[
B_1^{st}(X,Y)=B_1(X,Y)=Y^X
\]
for any topological space $Y$.
\end{proposition}

\begin{theorem} \label{t:B1-topological-sigma}
Let $Y$ be a perfectly normal space containing at least two points, $X$ be a $Y$-Tychonoff space, 
$\nn\in\mathfrak{N}$, and let $H$ be a subspace of $B(X,Y)$ containing $B_1^{st}(X,Y)$. Then $H$ is an $\nn\mbox{-}\sigma$-space if and only if $X$ is countable and $Y$ is an $\nn\mbox{-}\sigma$-space. In this case $H=Y^X$.
\end{theorem}

\begin{proof}
Assume that $H$ is an $\nn\mbox{-}\sigma$-space.
First we show that $X$ is countable. To this end, observe that every $\nn\mbox{-}\sigma$-space is a $\sigma$-space. Fix two distinct points $a,b\in Y$. Let us prove that   $X_{\aleph_0}$ has countable cellularity. Suppose for a contradiction that $c(X_{\aleph_0})$ is uncountable. Then there is an uncountable disjoint family $\V=\{ V_i:i\in I\}$ of zero-sets in $X$. By Lemma \ref{l:G-del-modif}, for every $i\in I$ there exist a zero-set $Z_i \subseteq V_i$ in $X$ and a function $f_i \in B_1^{st}(X,Y)\subseteq H$ such that
\[
f_i(Z_i)=\{ b\} \; \mbox{ and } \; f_i(X\SM Z_i)=\{ a\}.
\]
Set $K:= \{ f_i: i\in I\}\cup\{ \aaa\}$. Since the family $\V$ is disjoint, any standard neighborhood $W[\aaa;F,U]$ of the constant function $\aaa$ contains all but finitely many functions from $K$. Therefore, $K$ is a compact subset of $H$. Observe also that $K$ does not have countable base at $\aaa$ because the set $I$ is uncountable. Hence $K$ is not metrizable. Since every compact subset of a $\sigma$-space is metrizable (\cite[Corollary~4.7]{gruenhage}), we obtain that $H$ is not a $\sigma$-space. This contradiction shows that $c(X_{\aleph_0})=\aleph_0$.

Since $c(X_{\aleph_0})=\aleph_0$,  (ii) of Corollary \ref{c:cell-countable} implies that $X$ is countable. Now Proposition \ref{p:B1-countable} implies that  $B_1^{st}(X,Y)=Y^X$. Therefore $H=Y^X$. Thus $Y$ is an $\nn\mbox{-}\sigma$-space, see Corollary 5.6 of \cite{GK-GMS1}.

Conversely, assume that $X$ is countable and $Y$ is an $\nn\mbox{-}\sigma$-space. Then, by Proposition \ref{p:B1-countable}, $H=Y^X$ and hence $H$ is an $\nn\mbox{-}\sigma$-space by Corollary 5.6 of \cite{GK-GMS1}.\qed
\end{proof}

\begin{theorem} \label{t:B1-topological-cs}
Let $Y$ be a perfectly normal space containing at least two points, $X$ be a $Y$-Tychonoff space, 
$\nn\in\{ cp, ck,    cs^\ast\}$, and let $H$ be a subspace of $B(X,Y)$ containing $B_1^{st}(X,Y)$. Then $H$ has countable  $\nn$-character if and only if $X$ is countable and $Y$ has countable  $\nn$-character. In this case $H=Y^X$.
\end{theorem}

\begin{proof}
Assume that $H$  has countable  $\nn$-character.
First we show that $X$ is countable. To this end, observe that in all cases $X$ has countable  $cs^\ast$-character. Now we prove that   $X_{\aleph_0}$ has countable cellularity. Suppose for a contradiction that $c(X_{\aleph_0})$ is uncountable. Consider the compact subset $K$ of $H$ defined in the proof of Theorem \ref{t:B1-topological-sigma}. Then also $K$ being a subspace of $H$ has countable $cs^\ast$-character. However, by Proposition 9 of \cite{BZ}, the  $cs^\ast$-character of $K$ is uncountable. This contradiction shows that $c(X_{\aleph_0})=\aleph_0$.

Since $c(X_{\aleph_0})=\aleph_0$,  (ii) of Corollary \ref{c:cell-countable} implies that $X$ is countable. Now Proposition \ref{p:B1-countable} implies that   $B_1^{st}(X,Y)=Y^X$. Therefore $H=Y^X$. Thus $Y$ has countable  $\nn$-character, see Corollary 5.5 of \cite{GK-GMS1}.

Conversely, assume that $X$ is countable and $Y$ has countable  $\nn$-character. Then, by Proposition \ref{p:B1-countable}, $H=Y^X$, and hence $H$ has countable  $\nn$-character by Corollary 5.5 of \cite{GK-GMS1}.\qed
\end{proof}

Theorems \ref{t:B1-topological-sigma} and \ref{t:B1-topological-cs} immediately imply

\begin{corollary} \label{c:B1-topological-sigma-first}
Let $Y$ be a metric space containing at least two points, $X$ be a $Y$-Tychonoff space, and let $H$ be a subspace of $B(X,Y)$ containing $B_1^{st}(X,Y)$. Then the following assertions are equivalent:
\begin{enumerate}
\item[{\rm (i)}] $H$ is metrizable and $H=Y^X$;
\item[{\rm (ii)}] $H$ is a $\sigma$-space;
\item[{\rm (iii)}] $H$ has countable $cs^\ast$-character;
\item[{\rm (iv)}] $X$ is  countable.
\end{enumerate}
\end{corollary}

\begin{corollary} \label{c:B1-cosmic}
Let $Y$ be a metric space containing at least two points, $X$ be a $Y$-Tychonoff space, and let $H$ be a subspace of $B(X,Y)$ containing $B_1^{st}(X,Y)$. Then $H$ is a cosmic space if and only if $X$ is a countable Tychonoff space and $Y$ is separable. Consequently, $H=Y^X$ is a separable meric space and therefore $H$ is a $\Pp_0$-space.
\end{corollary}

\begin{proof}
Assume that $H$ is cosmic. Then, by Corollary \ref{c:B1-topological-sigma-first},  $X$ is countable and $H=Y^X$ is a metric space. Being cosmic the space $H$ and hence also $Y$ must be separable. Conversely, if $X$ is countable and $Y$ is separable, then  Corollary \ref{c:B1-topological-sigma-first} implies that $H=Y^X$. Therefore $H$ is separable and  metrizable. Thus $H$ is a $\Pp_0$-space and hence cosmic. \qed
\end{proof}

To obtain the Fr\'{e}chet--Urysohness of spaces of Baire function  we need the next proposition.
\begin{proposition} \label{p:G-delta-D-Tychonoff}
Let $Y$ be a perfectly normal space containing at least two points, $X$ be a $Y$-Tychonoff space, 
$\mathcal{D}$ be a finite subset of $Y$, and let $H$ be a subspace of $B(X,Y)$ containing $B_1^{st}(X,Y)$. Then $H$ considered as a subspace of $C_p(X_{\aleph_0},Y)$ has the following property:
 \begin{enumerate}
\item[{\rm (i)}] $H\cap C_p(X_{\aleph_0},\mathcal{D})$ is a relatively $\mathcal{D}_p$-Tychonoff subspace of $C_p(X_{\aleph_0},\mathcal{D})$.
\end{enumerate}
\end{proposition}

\begin{proof}
Since $X$ is Tychonoff (see Proposition \ref{p:Y-good=>good}), Proposition \ref{p:B(X,Y)-C(X,Y)} implies that $H$ is a  subspace of $C_p(X_{\aleph_0},Y)$. So it suffices to prove that $B_1^{st}(X,Y)$ satisfies (i).

Fix a closed subset $A$ of $X_{\aleph_0}$, a point  $y_0\in\mathcal{D}$  and a function $f:F\to \mathcal{D}$ defined on a finite subset  $F=\{x_1,\dots,x_n\}$ of $X\setminus A$. Since $X$ is Tychonoff, there are disjoint zero-sets $V_1,\dots,V_n$ in $X$ such that  $x_i\in V_i$ for every $i=1,\dots,n$. Choose a zero-set $V$ in $X$  such that $F\subseteq V$ and $V\cap A=\emptyset$. For every $i=1,\dots,n$, set $U_i :=V_i \cap V$. Then $U_1,\dots,U_n$ are disjoint zero-sets  in $X$  such that $A\cap U_i =\emptyset$ and $x_i\in U_i$ for every $i=1,\dots,n$.
By Lemma \ref{l:G-del-modif}, there exist a zero-set $Z$ in $X$ and a function ${\bar f}\in B_1^{st}(X,Y)\subseteq C_p(X_{\aleph_0},Y)$ such that $F\subseteq Z\subseteq \bigcup_{i=1}^n U_i$ and
\[
{\bar f}{\restriction}_F = f, \; {\bar f}(X\SM Z) \subseteq \{ y_0\} \; \mbox{ and } \; {\bar f}(X)\subseteq \{y_0\}\cup f(F) \subseteq \mathcal{D}.
\]
Taking into account that $A\subseteq X\SM Z$, we obtain that ${\bar f}$ is a desired extension of $f$. \qed
\end{proof}


\begin{theorem} \label{t:B1-topological-sequential}
Let $Y$ be a metrizable space containing at least two points, $X$ be a $Y$-Tychonoff space, and let $H$ be a subspace of $B(X,Y)$ containing $B_1^{st}(X,Y)$. Then the following assertions are equivalent:
\begin{enumerate}
\item[{\rm (i)}] $H$ is a Fr\'{e}chet--Urysohn space;
\item[{\rm (ii)}] $H$ is a sequential space;
\item[{\rm (iii)}] $H$ has countable tightness;
\item[{\rm (iv)}] $X_{\aleph_0}$ is a Lindel\"{o}f space;
\item[{\rm (v)}] $X_{\aleph_0}$ has the property $\gamma$.
\end{enumerate}
Moreover, if $X$ is scattered, then {\em (i)-(v)} are equivalent to
\begin{enumerate}
\item[{\rm (vi)}] $X$ is a Lindel\"{o}f space.
\end{enumerate}
\end{theorem}

\begin{proof}
The implications (i)$\Rightarrow$(ii)$\Rightarrow$(iii) are clear.
\smallskip

(iii)$\Rightarrow$(iv) Let $\xi=\{ V_i: i\in I\}$ be an open cover of $X_{\aleph_0}$. By the definition of the Baire topology $\tau_b$, we can assume that all $V_i$ are zero-sets in $X$. Moreover we can assume that $\xi$ is closed under taking finite unions. Fix two distinct points $a,b\in Y$. By Lemma \ref{l:G-del-modif}, for every $i\in I$ and each finite subset $F\subseteq V_i$ there are a zero-set $Z_{i,F}$ and a  function $f_{i,F}\in B_1^{st}(X,Y)\subseteq H$ such that
\begin{equation} \label{equ:B1-general-1}
Z_{i,F}\subseteq V_{i}, \;\; f_{i,F}(Z_{i,F})=\{ a\} \; \mbox{ and } \; f_{i,F}(X\SM Z_{i,F})=\{ b\}.
\end{equation}

Since the cover $\xi$ is closed under taking finite unions, the constant function $\aaa$ belongs to the closure of the family $\FF=\{ f_{i,F}: i\in I, F\in [V_i]^{<\w}\}$. As $H$ has countable tightness, there is a sequence $S=\{ f_{i_n,F_n}: n\in\w\}$ in $\FF$ such that $\aaa\in \overline{S}$. Set $X_0:=\bigcup_{n\in \w} V_{i_n}$. We claim that $X_0=X$. Indeed, assuming the converse we can find a point $z\in X\SM X_0$. Then, by (\ref{equ:B1-general-1}), $f_{i_n,F_n}(z)=b$ for every $n\in\w$. Choose an open neighborhood $U$ of $a$ such that $b\not\in U$. Then
$
W[\aaa;\{z\},U] \cap S =\emptyset
$
that contradicts the inclusion $\aaa\in \overline{S}$. Thus $X_0=X$ and hence $X_{\aleph_0}$ is Lindel\"{o}f.
\smallskip

(iv)$\Rightarrow$(v) Since $X_{\aleph_0}$ is a $P$-space and  Lindel\"{o}f, the Galvin lemma \cite{GN} implies that $X_{\aleph_0}$ has the property $\gamma$.
\smallskip

(v)$\Rightarrow$(i)  First we note that the space $X_{\aleph_0}$ is $Y$-Tychonoff by Proposition \ref{p:G-del-modif}. Since $X_{\aleph_0}$ has the property $\gamma$, Theorem \ref{t:FU-Cp} implies that the space $C_p(X_{\aleph_0},Y)$ is  Fr\'{e}chet--Urysohn. As $H\subseteq B(X,Y) \subseteq C_p(X_{\aleph_0},Y)$ it follows that also $H$ is a Fr\'{e}chet--Urysohn space.
\smallskip

Finally, if $X$ is scattered, then (iv) and (vi) are equivalent by a result of Uspenski\u{\i} (see the proof of Lemma II.7.14 of \cite{Arhangel}) which states that a scattered space $X$ is Lindel\"{o}f if and only if $X_{\aleph_0}$ is Lindel\"{o}f. \qed
\end{proof}

\begin{corollary} \label{c:B1-topological-sequential}
Let $Y$ be a metrizable space containing at least two points, $X$ be a $Y$-Tychonoff space, and let $H$ be a subspace of $B(X,Y)$ containing $B_1^{st}(X,Y)$. Then the following assertions are equivalent:
\begin{enumerate}
\item[{\rm (i)}] $H^{k}$ is Fr\'{e}chet--Urysohn for each $k\in (0,\omega]$;
\item[{\rm (ii)}] $H^k$ is a sequential space for each $k\in (0,\omega]$;
\item[{\rm (iii)}] $H^k$ has countable tightness for each $k\in (0,\omega]$;
\item[{\rm (iv)}] $X_{\aleph_0}$ has the property $\gamma$.
\item[{\rm (v)}] $H^k$ is Fr\'{e}chet--Urysohn (sequential or has countable tightness) for some $k\in (0,\omega]$.
\end{enumerate}
\end{corollary}

\begin{proof}
The implications (i)$\Rightarrow$(ii)$\Rightarrow$(iii) and (i)$\Rightarrow$(v) ((ii)$\Rightarrow$(v) and (iii)$\Rightarrow$(v), respectively) are clear.

(v)$\Rightarrow$(iv) Since $H$ is a closed subspace of $H^k$, also the space $H$ is Fr\'{e}chet--Urysohn (sequential or has countable tightness). Therefore, by Theorem \ref{t:B1-topological-sequential}, $X_{\aleph_0}$ has the property $\gamma$.

(iv)$\Rightarrow$(i) Since $X_{\aleph_0}$ has the property $\gamma$, Theorem \ref{t:Cp(X,Y)-FU} implies that the space $C_p(X_{\aleph_0},Y)^\w$ is  Fr\'{e}chet--Urysohn. As $H\subseteq B(X,Y) \subseteq C_p(X_{\aleph_0},Y)$ it follows that $H^k$ embeds into $C_p(X_{\aleph_0},Y)^\w$ for every $k\in (0, \omega]$. Thus $H^k$ is a Fr\'{e}chet--Urysohn space. \qed
\end{proof}

\begin{example} \label{exa:B1-scatterd}{\em
Let $X=[0,\w_1]$ and $Z=[0,\w_1)$. Note that $X$ and $Z$ are scattered spaces. Since $X$ and $Z$ are zero-dimensional $T_1$-spaces (\cite[6.2.18]{Eng}), by Proposition \ref{p:Y-z-Tychonoff-exa}(iii), $X$ and $Z$ are $Y$-$z$-normal for any topological space $Y$. By  Theorem \ref{t:B1-topological-sequential},  for each metrizable space $Y$ containing at least two points (for example $Y=\IR$ or $Y=\mathbf{2}$) and every ordinal $\alpha>0$, we obtain

(i) $B_\alpha^{st}\big(X,Y\big)$ and $B_\alpha\big(X,Y\big)$ are Fr\'{e}chet--Urysohn spaces, but

(ii) $B_\alpha^{st}\big(Z,Y\big)$ and $B_\alpha\big(Z,Y\big)$ have uncountable tightness.  \qed}
\end{example}

In the general case the  tightness of spaces of Baire functions is computed in the next proposition which  generalizes (B) of Pestryakov's Theorem \ref{t:Pest-k-space} (see also \cite{os}).
\begin{proposition} \label{p:B1(X,Y)-tight}
Let $Y$ be a metrizable space containing at least two points, $X$ be a $Y$-Tychonoff space, and let $H$ be a subspace of $B(X,Y)$ containing $B_1^{st}(X,Y)$. Then
\[
t(H)=p\mbox{-}\Lin(X_{\aleph_0})=\sup\{l(X^n_{\aleph_0}): n\in \omega \}.
\]
\end{proposition}

\begin{proof}
By Proposition \ref{p:G-delta-D-Tychonoff}, $H$ is  a relatively $Y_p$-Tychonoff subspace of $C_p(X_{\aleph_0},Y)$. Then the first equality $t(H)=p\mbox{-}\Lin(X_{\aleph_0})$ follows from Theorem \ref{t:Cp-tight}. Now the second equality immediately follows from Theorem 4.7.1 and Corollary 4.7.3 of \cite{mcoy} applied to $C_p(X_{\aleph_0})$. \qed
\end{proof}


For compact spaces we have the following result.
\begin{theorem} \label{t:B1-scattered}
Let $Y$ be a metrizable space containing at least two points, $K$ be a $Y$-Tychonoff compact space, and let $H$ be a subspace of $B(K,Y)$ containing $B_1^{st}(K,Y)$. Then the following assertions are equivalent:
\begin{enumerate}
\item[{\rm (i)}] $H$ is a Fr\'{e}chet--Urysohn space;
\item[{\rm (ii)}] $H$ is a sequential space;
\item[{\rm (iii)}] $H$ has countable tightness;
\item[{\rm (iv)}] $K$ is scattered.
\end{enumerate}
\end{theorem}

\begin{proof}
The implications (i)$\Rightarrow$(ii)$\Rightarrow$(iii) are clear.
\smallskip

(iii)$\Rightarrow$(iv) Suppose for a contradiction that $K$ is not scattered. Then, by Theorem  8.5.4 of \cite{sema},  there is a continuous surjective map $T:K\to \II=[0,1]$. Since $\II$ is a normal space, $K=\bigcup_{t\in\II} T^{-1}(t)$ is a continual disjoint union of zero-sets. Therefore the space $K_{\aleph_0}$ is not Lindel\"{o}f. Now Theorem \ref{t:B1-topological-sequential} implies that the tightness of $H$ is  uncountable, a contradiction.
\smallskip

(iv)$\Rightarrow$(i) Since $K$ is a compact scattered space, Theorem 5.7 of \cite{leri} states that $K_{\aleph_0}$ is Lindel\"{o}f and  Theorem \ref{t:B1-topological-sequential} applies. \qed
\end{proof}

\begin{theorem} \label{t:B1-k-space}
Let $Y$ be a non-compact metrizable space, $X$ be a $Y$-Tychonoff space, and let $H$ be a subspace of $B(X,Y)$ containing $B_1^{st}(X,Y)$. Then the following assertions are equivalent:
\begin{enumerate}
\item[{\rm (i)}] $H$ is a Fr\'{e}chet--Urysohn space;
\item[{\rm (ii)}] $H$ is a sequential space;
\item[{\rm (iii)}] $H$ is a  $k$-space;
\item[{\rm (iv)}] $H$ has countable tightness;
\item[{\rm (v)}] $X_{\aleph_0}$ satisfies the property $\gamma$.
\end{enumerate}
\end{theorem}

\begin{proof}
Observe that $H \subseteq B(X,Y) \subseteq C_p(X_{\aleph_0},Y)$. Since $Y$ is a metric space, Proposition \ref{p:G-del-modif} implies that the space $X_{\aleph_0}$ is $Y$-Tychonoff. By Proposition \ref{p:G-delta-D-Tychonoff},  for every countable (finite or not) subset $\mathcal{D}\subseteq Y$, $H\cap C_p(X_{\aleph_0},\mathcal{D})$ is a relatively $\mathcal{D}_p$-Tychonoff subspace of $C_p(X_{\aleph_0},\mathcal{D})$. Now Theorems \ref{t:H-k-space} and \ref{t:B1-topological-sequential} apply. \qed
\end{proof}

Since  the class of Tychonoff spaces coincides with the class of $\IR$-Tychonoff spaces, Theorem \ref{t:B1-k-space} implies the following strengthening of Pestryakov's Theorem \ref{t:Pest-k-space}.

\begin{corollary} \label{c:B1-k-space}
Let $X$ be a Tychonoff space, and let $H$ be a subspace of $B(X)$ containing $B_1^{st}(X)$. Then the following assertions are equivalent:
\begin{enumerate}
\item[{\rm (i)}] $H$ is a Fr\'{e}chet--Urysohn space;
\item[{\rm (ii)}] $H$ is a sequential space;
\item[{\rm (iii)}] $H$ is a  $k$-space;
\item[{\rm (iv)}] $H$ has countable tightness;
\item[{\rm (v)}] $X_{\aleph_0}$ satisfies the property $\gamma$.
\end{enumerate}
\end{corollary}

If $X$ has countable pseudocharacter, then $X_{\aleph_0}$ satisfies the property $\gamma$ if and only if $X$ is countable because, by Corollary \ref{c:cell-countable}, $X_{\aleph_0}$ is discrete. Therefore, for subspaces of $B(X,Y)$, Theorem \ref{t:B1-k-space} and Corollary \ref{c:B1-topological-sigma-first} imply the following extension of Theorem \ref{t:Baire-class}.

\begin{corollary} \label{c:B1-metric-space}
Let $Y$ be a non-compact metrizable space, $X$ be a $Y$-Tychonoff space of countable pseudocharacter, and let $H$ be a subspace of $B(X,Y)$ containing $B_1^{st}(X,Y)$. Then the following assertions are equivalent:
\begin{enumerate}
\item[{\rm (i)}] $H$ is metrizable and $H=Y^X$;
\item[{\rm (ii)}] $H$ is a  $k$-space;
\item[{\rm (iii)}] $H$ has countable tightness;
\item[{\rm (iv)}] $H$ is a $\sigma$-space;
\item[{\rm (v)}] $H$ has countable $cs^\ast$-character;
\item[{\rm (vi)}] $X$ is countable.
\end{enumerate}
\end{corollary}

Now we consider bounded Baire functions. Let $X$ be a Tychonoff space and $E$ be a locally convex space. A map $f:X\to E$ is called {\em bounded} ({\em relatively compact}) if the image $f(X)$ is a bounded (respectively, relatively compact) subset of $E$. For every countable ordinal $\alpha$, we denote by $B^{b}_\alpha(X,E)$, $B^{rc}_\alpha(X,E)$, $B^{st,b}_\alpha(X,E)$ or $B^{st,rc}_\alpha(X,E)$ the family of all functions from $B_\alpha(X,E)$ or $B^{st}_\alpha(X,E)$ which are bounded or relatively compact, respectively. It is clear that
\[
B^{st,rc}_\alpha(X,E) \subseteq B^{st,b}_\alpha(X,E) \subseteq B^{st}_\alpha(X,E) \; \mbox{ and } \; B^{rc}_\alpha(X,E) \subseteq B^{b}_\alpha(X,E) \subseteq B_\alpha(X,E).
\]

For the special case when $Y$ is a metrizable locally convex space, the next result strengthens Corollary \ref{c:B1-topological-sigma-first} and Theorem \ref{t:B1-k-space}.
\begin{theorem} \label{t:Baire-bounded}
Let $E$ be a metrizable locally convex space, $X$ be a Tychonoff space, and let $H$ be a subspace of $B(X,E)$ containing $B_1^{st,rc}(X,E)$. Then:
\begin{enumerate}
\item[{\rm (A)}] The following assertions are equivalent:
\begin{enumerate}
\item[{\rm (i)}] $H$ is metrizable;
\item[{\rm (ii)}] $H$ is a $\sigma$-space;
\item[{\rm (iii)}] $H$ has countable $cs^\ast$-character;
\item[{\rm (iv)}] $X$ is  countable.
\end{enumerate}
\item[{\rm (B)}]
 The following assertions are equivalent:
\begin{enumerate}
\item[{\rm (i)}] $H$ is a Fr\'{e}chet--Urysohn space;
\item[{\rm (ii)}] $H$ is a sequential space;
\item[{\rm (iii)}] $H$ is a  $k$-space;
\item[{\rm (iv)}] $H$ has countable tightness;
\item[{\rm (v)}] $X_{\aleph_0}$ satisfies the property $\gamma$.
\end{enumerate}
\end{enumerate}
\end{theorem}

\begin{proof}
The proof is actually given in the proofs of Theorems \ref{t:B1-topological-sigma}, \ref{t:B1-topological-cs}, \ref{t:B1-topological-sequential} and Proposition \ref{p:G-delta-D-Tychonoff}. Indeed, all functions used there have only finite image and hence belong to $B_1^{st,rc}(X,E)$.\qed
\end{proof}

We finish this section with the following analogue of Proposition \ref{p:B1-countable}.
\begin{proposition} \label{p:B1-countable-bounded}
Let $E$ be a locally convex space and let $X$ be a countable Tychonoff space. Then a map $f:X\to E$ is bounded (relatively compact) if and only if it belongs to $B_1^{st,b}(X,E)$ (respectively, $B_1^{st,rc}(X,E)$).
\end{proposition}

\begin{proof}
The assertion is trivial if $X$  is finite. So, we assume that $X$ is infinite. The sufficiency follows from the definition of $B_1^{st,b}(X,E)$ ($B_1^{st,rc}(X,E)$). To prove the necessity, let $f:X\to E$ be a bounded (relatively compact) map. Write $X$ as the union $X=\bigcup_{n\in\w}X_n$ of an increasing sequence $\{X_n\}_{n\in\w}$ of nonempty finite subsets $X_n$ of $X$. Since $X$ is countable and Tychonoff, it is zero-dimensional by Corollary 6.2.8 of \cite{Eng}. For every $n\in\w$, by Proposition \ref{p:2-Y-I-Tychonoff},  the identity map $s_n:F_n\to F_n$ can be extended to a continuous function  $r_n:X\to X_n$. It follows that the function sequence $\{f\circ r_n\}_{n\in\w}$ stably converges to $f$ and hence $f\in B_1^{st}(X,Y)$. Since $f(X)$ is a bounded (respectively, relatively compact) subset of $E$, we obtain $f\in B_1^{st,b}(X,E)$ (respectively, $f\in B_1^{st,rc}(X,E)$). \qed
\end{proof}


\section{Normality of spaces of Baire one functions} \label{sec:normal}


In \cite[3.1.H]{Eng}, Engelking showed that the space $\w^\mathfrak{c}$ contains a closed discrete subspace of cardinality $\mathfrak{c}$, where $\w$ is endowed with the discrete topology. The next proposition essentially generalizes this fact (just apply the proposition to $Y=\w$, $X=2^\w$ and $H=Y^X$).

\begin{proposition} \label{p:Eng-1}
Let $Y$ be a Tychonoff space containing a closed and discrete subspace $D=\{ y_n\}_{n\in\w}$, $X$  be an uncountable $Y$-$z$-Tychonoff metrizable compact  space, and let $H$ be a subspace of $Y^X$ containing $B^{st}_1(X,Y)$. Then $H$ contains a discrete and closed subspace $\mathcal{F}\subseteq B^{st}_1(X,Y)$ of cardinality $\mathfrak{c}$ such that $f(X)\subseteq D$ for every $f\in\mathcal{F}$ (in fact, $\mathcal{F}$ is closed and discrete in $Y^X$).
\end{proposition}

\begin{proof}
First we note that the metrizable space $X$ is $Y$-normal by Proposition \ref{p:z-normal-Y-normal}.
Fix a metric $\rho$ on $X$ such that $\rho(x,y)\leq 1$ for all points $x,y\in X$. For every $t\in X$, similar to \cite[3.1.H]{Eng}, define a function $f_t:X\to D$ as follows:
\[
f_t(t):=y_0, \mbox{ and } \; f_t(x):= y_n \mbox{  if } \; \tfrac{1}{n+1}< \rho(x,t)\leq \tfrac{1}{n}.
\]

{\em Claim 1.  $f_t\in B_1^{st}(X,Y)$ for every $t\in X$.} Indeed, fix $t\in X$. For every natural numbers $n,k$ such that $1\leq n\leq k$, set
\[
U^t_{n,k} := \left\{ x\in X: \tfrac{1}{n+1}+\tfrac{1}{4k^2}\leq \rho(x,t)\leq \tfrac{1}{n}\right\}
\]
and
\[
V^t_k := \left\{ x\in X: \rho(x,t)\leq \tfrac{1}{k+1}\right\}.
\]
It is easy to see that the family $\{ U^t_{n,k}\}_{n=1}^k \cup\{ V^t_k\}$ of closed subsets of $X$ is disjoint and $t\in V^t_k$. Since $X$ is a $Y$-normal space, for every $k\geq 1 $, there exists a continuous function $f_{t,k}:X\to Y$ such that
\[
f_{t,k}(x)=y_n \; \mbox{ if }\; x\in U^t_{n,k} \mbox{ for }\; 1\leq n\leq k, \; \mbox{ and }\; f_{t,k}(x)=y_0 \;\mbox{ if }\; x\in V^t_k.
\]
To prove the claim we shall show that $f_{t,k}$ stably converges to $f_{t}$. Indeed, $f_{t,k}(t)=y_0$ for every $k\geq 1$. Now let $x\not=t$ and choose $n\geq 1$ such that $\tfrac{1}{n+1}< \rho(x,t)\leq \tfrac{1}{n}$. Take $k_0\geq 1$ such that $\tfrac{1}{n+1}+\tfrac{1}{4k^2_0}< \rho(x,t)$. Then for every $k\geq\max\{ n,k_0\}$, we have $x\in U^t_{n,k}$ and hence $f_{t,k}(x)=y_n=f_t(x)$. Thus $f_t\in B_1^{st}(X,Y)$. Claim 1 is proved.

Let $\mathcal{F}:=\{ f_t:t\in X\}$. We shall show that $\mathcal{F}$ satisfies the conditions of the proposition. First we show that $\mathcal{F}$ is discrete in $Y^X$. To this end, consider two  distinct points $t,t'\in X$. Then $f_t(t)=y_0$ but $f_{t'}(t)\in D\SM \{ y_0\}$. If $\mathcal{O}_0$ is an open neighborhood of $y_0$ such that $\mathcal{O}_0\cap D=\{ y_0\}$, then $W[ f_t; \{t\},\mathcal{O}_0]\cap \mathcal{F} =\{ f_t\}$. Thus $\mathcal{F}$ is discrete.

To show that $\mathcal{F}$ is closed, we shall prove that $\mathcal{F}$ is closed  in $Y^X$. Suppose for a contradiction that $\mathcal{F}$ is not closed in $Y^X$. Fix a function $\chi\in \overline{\mathcal{F}}\SM \mathcal{F}$. Observe that $\chi(X)\subseteq D$ because $D$ is closed and, by construction, $f(X)\subseteq D$ for every $f\in \mathcal{F}$.
\smallskip

{\em Claim 2. $\chi(X)\subseteq D\SM\{ y_0\}$.} Indeed, assuming the converse we can find a point $t\in X$ such that $\chi(t)=y_0$. Since $\chi\not= f_t$, there is an $x\in X$ such that $\chi(x)\not=f_t(x)$. Choose an open neighborhood $U_{\chi(x)}$ of $\chi(x)$ in $Y$ such that $U_{\chi(x)}\cap D=\{\chi(x)\}$. Then
\[
\big( W\big[ \chi; \{ t\},  \mathcal{O}_0\big] \cap W\big[ \chi; \{ x\}, U_{\chi(x)}\big]\big) \cap \mathcal{F} =\emptyset
\]
that contradicts the inclusion $\chi\in \overline{\mathcal{F}}$. Thus $\chi(X)\subseteq D\SM\{ y_0\}$.
\smallskip

{\em Claim 3. There is an $m\geq 1$ such that $\chi(X)\subseteq \{ y_1,\dots,y_m\}$.} Indeed, suppose for a contradiction that there exist a sequence $\{ x_k\}_{k\geq 1}\subseteq X$ and an increasing sequence $\{ n_k\}_{k\geq 1}\subseteq\w$ such that $\chi(x_k)=y_{n_k}$. Since $X$ is compact, without loss of generality we can assume that $x_k \to x_0$. Note that $\chi(x_0)=y_{n_0}$ for some $n_0>0$. Choose an open neighborhood $\mathcal{O}_{n_0}$ of $y_{n_0}$ such that $\mathcal{O}_{n_0}\cap D=\{ y_{n_0}\}$. Since $\chi\in \overline{\mathcal{F}}$, the standard neighborhood $W[\chi; \{x_0\},\mathcal{O}_{n_0}]$ of $\chi$ contains only those $f_t\in \mathcal{F}$ for which $f_t(x_0)=y_{n_0}$. Denote by $T$ the set of all $t\in X$ such that
\[
\tfrac{1}{n_0+1}< \rho(x_0,t)\leq \tfrac{1}{n_0},
\]
so $t\in T$ if and only if $f_t(x_0)=y_{n_0}$. As the metric $\rho$ is continuous and $X$ is compact, there is an open neighborhood $V$ of $x_0$ such that
\begin{equation} \label{equ:Baire-1-normal-1}
\tfrac{1}{n_0+2}< \rho(x,t)\leq \min\left\{ \tfrac{1}{n_0-1},1\right\}
\end{equation}
for every $x\in V$ and each $t\in T$. But the definition of $f_t$ and (\ref{equ:Baire-1-normal-1}) imply that
\begin{equation} \label{equ:Baire-1-normal-2}
f_t(x_k)\in \big\{ y_{\max\{ n_0-1,1\}}, y_{n_0}, y_{n_0+1} \big\}
\end{equation}
for every $t\in T$ and all $k\in\w$ for which $x_k\in V$. Choose $k\in\w$ such that $k>n_0+1$ and $x_k\in V$, and take an open neighborhood $\mathcal{O}_{n_k}$ of $y_{n_k}$ such that $\mathcal{O}_{n_k}\cap D=\{ y_{n_k}\}$. Then (\ref{equ:Baire-1-normal-2}) implies
\[
\big( W\big[ \chi; \{x_0\}, \mathcal{O}_{n_0}\big] \cap W\big[ \chi; \{x_k\}, \mathcal{O}_{n_k}\big] \big)\cap \mathcal{F} =\emptyset
\]
which contradicts the inclusion $\chi\in \overline{\mathcal{F}}$. Claim 3 is proved.
\smallskip

Since $X$ is compact, choose a finite subset $\{z_1,\dots,z_s\}$ of $X$ such that for every $x\in X$  there is an $i_x\in\{ 1,\dots,s\}$ for which $\rho(x,z_{i_x})<\tfrac{1}{2m}$, where $m\in \w$  is defined in Claim 3. Choose open neighborhoods $O_1,\dots,O_s$ of $\chi(z_1),\dots,\chi(z_s)$, respectively, such that $O_i\cap D=\{ \chi(z_i)\}$ for every $i=1,\dots,s$. Consider the standard neighborhood
\[
W=\bigcap_{ i=1}^s W\big[\chi; \{z_i\}, O_i\big]
\]
of the function $\chi$. Now, for every $t\in X$, choose  $i_t\in\{ 1,\dots,s\}$ such that $\rho(t,z_{i_t})<\tfrac{1}{2m}$. Then either $t=z_{i_t}$ and hence $f_t(z_{i_t})=y_0$, or $t\not= z_{i_t}$ and hence $f_t(z_{i_t})=y_{n_t}$ for some $n_t >m$. So, by Claims 2 and 3 and the construction of $W$, it follows that $f_t\not\in W$ for every $t\in X$. Therefore $\chi\not\in \overline{\mathcal{F}}$ that contradicts the choice of the function $\chi$. Thus the family $\mathcal{F}$ is closed in $Y^X$.\qed
\end{proof}


To show that a given space is not normal we shall use the following Jones' Lemma, see \cite[Lemma~3.5]{Hodel}.
\begin{lemma} \label{l:discrete-normal}
If $X$ is a normal space, then $2^{|D|}\leq 2^{d(X)}$ for every closed discrete $D\subseteq X$. In particular, if $X$ is normal and separable, then: {\em (1)} $X$ cannot have a closed discrete set of cardinality $\geq \mathfrak{c}$; {\em (2)} $2^\w <2^{\w_1}$ implies that $X$ cannot have a closed discrete set of cardinality $\w_1$.
\end{lemma}



\begin{proposition} \label{p:compact-normal-B1}
Let $Y$ be a separable non-compact metric space, $X$ be a $Y$-$z$-Tychonoff metrizable compact space, and let $H$ be a subspace of $Y^X$ containing $B^{st}_1(X,Y)$. Then $H$ is a normal space if and only if $X$ is countable. In this case $H=Y^X$.
\end{proposition}

\begin{proof}
Let $H$ be a normal space. If $X$ is uncountable, then, by Proposition \ref{p:Eng-1}, $H$ contains a closed and discrete subspace $\mathcal{F}$ of cardinality $\mathfrak{c}$.
Since $X$ is $Y$-Tychonoff (Proposition \ref{p:Y-z-Tychonoff-exa}),  Corollary \ref{c:density-C(X,Y)} implies that the space $C_p(X,Y)$ is  separable. Since $C_p(X,Y)$ is dense in $H$, the space $H$ is also separable. Therefore, by Lemma \ref{l:discrete-normal}, $H$ is not normal. This contradiction shows that $X$ must be countable. Conversely, if $X$ is countable, then, by Proposition \ref{p:B1-countable}, $H=Y^X$, and therefore $H$ is metrizable and hence normal. \qed
\end{proof}

Recall that a topological space $X$ is called {\em $k$-scattered} if every compact Hausdorff subspace of $X$ is scattered. It is proved in \cite[p.34]{CD} that a \v{C}ech-complete space $X$ is scattered if and only if it is $k$-scattered (for a short proof of this result see Theorem 8.7 in \cite{BG-Baire}).

\begin{proposition} \label{p:B1-normal}
Let $Y$ be a path-connected non-compact separable metric space, and let $X$ be a Tychonoff space. Assume that $H$ is a subspace of $B(X,Y)$ containing $B_1^{st}(X,Y)$. If $H$ is a normal space, then $X$ is $k$-scattered.
\end{proposition}

\begin{proof}
We have to show that every compact subspace of $X$ is scattered. Assuming the converse we can find a non-scattered compact subspace $K$ of $X$. By Theorem 8.5.4 of \cite{sema}, there is a continuous surjective map $h:K\to \II=[0,1]$. The Tietze--Urysohn Theorem \cite[2.1.8]{Eng} implies that $h$ has a continuous (surjective) extension   $g:X\to \II$. Observe that the adjoint map $g^\ast:Y^\II \to Y^X$, $g^\ast(f):=f\circ g$, is a homeomorphism of $Y^\II$ onto a {\em closed} subspace of $Y^X$.
By Proposition \ref{p:Y-z-Tychonoff-exa}, $\II$ is $Y$-$z$-Tychonoff. Since $Y$ is not compact, it contains a closed and discrete countably infinite subset. Now we can apply Proposition \ref{p:Eng-1} to get that the space  $B_1^{st}(\II,Y)$ contains a subset $A$ of cardinality $\mathfrak{c}$ which is closed and discrete in $Y^\II$. Since $g$ is continuous we also have $g^\ast\big(B_1^{st}(\II,Y)\big) \subseteq B_1^{st}(X,Y)$. Therefore $g^\ast(A)$ is also a closed and discrete subset of $H$. Note that a closed subset of a normal space is a normal space. Therefore the space  $Z:=\overline{g^\ast\big(B_1^{st}(\II,Y)\big)}^{H}$ is normal. Since $B_1^{st}(\II,Y)$ is separable (because it contains a dense subspace $C_p(\II,Y)$ which is separable by Corollary \ref{c:density-C(X,Y)}), we obtain that also $Z$ is separable.  Therefore $g^\ast(A)$ is a closed and discrete subset of the separable space $Z$ and has cardinality $\mathfrak{c}$. Now Lemma \ref{l:discrete-normal} implies that $Z$ is not normal, a contradiction. \qed
\end{proof}

Recall (see \cite{leri}) that a Tychonoff space $X$ is called {\em functionally countable} if the set $f(X)$ is countable for each function $f\in C(X)$. In \cite{BG-Baire}, topological spaces with this property are called $\IR$-countable. In Theorem 4.1 of \cite{Choban}, Choban proved that a Tychonoff space $X$ is functionally countable if and only if  $f(X)$ is countable for each function $f\in B_1(X)$. Proposition 4.4 of \cite{Choban} states that if $X_{\aleph_0}$ is a Lindel\"{o}f space, then $X$ and $X_{\aleph_0}$ are functionally countable spaces.

Recall that a subset $A$ of a topological space $X$ is called a {\em $Z_{\sigma}$-set in $X$}  if $A=\bigcup \{Z_i: i\in \omega\}$, where $Z_i$ is a zero-set  of $X$ for each $i\in \omega$.
For topological spaces $X$ and $Y$, let $\AAA^0_1(X,Y)$ be the spaces of all functions $f:X\to Y$ such that $f^{-1}(U)$ is a $Z_{\sigma}$-set in $X$.
We need the following assertion.

\begin{proposition} \label{p:B1=B2}
Let $Y$ be a perfectly normal space containing at least two points, and let $X$ be a $Y$-Tychonoff space such that $B_1(X,Y)=\A_1^0(X,Y)$. If $X$ is functionally countable, then $B_1(X,Y)=B(X,Y)$. In particular, $B_1(X,Y)=B_2(X,Y)$.
\end{proposition}

\begin{proof}
Let $f\in B(X,Y)$. Since $Y$ is perfectly normal, for every open subset $U \subseteq Y$ there is a continuous function $g:Y\to \IR$ such that $U= g^{-1}(\IR\SM \{0\})$. By induction on $\alpha\in\w_1$, one can check that the function $g\circ f:X\to \IR$ is also Baire. Hence $f^{-1}(U) = (g\circ f)^{-1} (\IR\SM \{0\})$ is a Baire subset of $X$. Therefore, by the equality $B_1(X,Y)=\A_1^0(X,Y)$, to prove the proposition it is sufficient to show that every Baire subset $A$ of $X$ is a $Z_{\sigma}$-set in $X$. So let $A$ be a Baire subset of $X$.
\smallskip

{\em Claim 1.  There are $g=(g_i)_{i\in \omega}\in C(X,\mathbb{R}^{\omega})$ and a Baire subset $H$ of $\mathbb{R}^{\omega}$ such  that $A=g^{-1}(H)$.} Indeed,
by the definition of  Baire sets of a topological space, the set $A$ is generated by a family $\{Z_i : i\in \omega\}$ of zero-sets of $X$, i.e. $A=\varphi [Z_0,..., Z_i,...]$ where $\varphi$ is an action in the hierarchy on the Baire sets. For every $i\in\w$, choose a function $g_i\in C(X)$ such that $Z_i$ is the zero-set of $g_i$ and consider the diagonal mapping $g=(g_i)_{i\in \omega}: X \rightarrow \mathbb{R}^{\w}$. For each $i\in \omega$, set $P_i :=\{x=(x_i)\in \mathbb{R}^{\w} : x_i=0 \}$. It is clear that  for each $i\in \omega$, $P_i$ is a zero-set of $\mathbb{R}^{\w}$ and  $g^{-1}(P_i)=Z_i$. Let $H=\varphi [P_1,..., P_i,...]$, so $H$ is a Baire subset of $\mathbb{R}^{\omega}$. Since the operation of taking preimage preserves unions, intersections and differences we obtain that $A=g^{-1}(H)$. The claim is proved.
\smallskip

As $X$ is functionally countable, all the sets $g_i(X)\subseteq \IR$ are countable. Therefore the product $\prod_{i\in\w} g_i(X)$ is a separable, zero-dimensional metrizable space, and hence so is its  subspace $g(X)$.  By Theorem 6.2.16 in \cite{Eng}, there is an embedding $p$ of $g(X)$ into $2^{\omega}$. We can assume that $2^{\omega}\subseteq \mathbb{R}$. As $X$ is functionally countable, the set $p\circ g(X)$ is countable. Since $p$ is a bijection, we obtain that $g(X)$ and hence also $H$ are countable. But every countable subset of $\mathbb{R}^{\omega}$ is a $Z_{\sigma}$-set in
$\mathbb{R}^{\omega}$. Hence $A=g^{-1}(H)$ is a $Z_{\sigma}$-set in $X$. \qed
\end{proof}



Let $X$ and $Y$ be sets. For a function $f:X\to Y$ and a subset  $T\subseteq Y$, we set
\[
\sigma(f,T) := \big\{ h\in \sigma(f): h(x)\in T \mbox{ for every } x\in X \mbox{ such that } h(x)\not= f(x)\big\}.
\]
In the next theorem we shall use the following theorem (although this result was proved in \cite{Gabr-B1} for the case when $Y$ is an abelian metrizable group containing an infinite uniformly discrete subset, exactly the same proof works also for non-compact metrizable spaces).
\begin{theorem}[\cite{Gabr-B1}] \label{t:Baire-class-normal}
Let $Y$ be a non-compact metrizable space. Choose a discrete and closed sequence $T=\{ g_n\}_{n\in\w}$ in $Y$. Assume that  $X$ is a set and let $H$ be a subspace of $Y^X$ containing pointwise limits of sequences from $\sigma(\mathbf{g}_1, T)\cup \sigma(\mathbf{g}_0,T)$. Then $H$ is a normal space if and only if $X$ is countable.
\end{theorem}

Recall that a topological space $X$ is called {\em $K$-analytic} if $X$ is a continuous image of a Lindel\"of \v Cech-complete space.

\begin{theorem} \label{t:B1-normal}
Let $Y$ be a path-connected non-compact separable metric space, and let $X$ be a Tychonoff $K$-analytic space  of countable pseudocharacter. If $B_1(X,Y)=\A_1^0(X,Y)$, then the following assertions are equivalent:
\begin{enumerate}
\item[{\rm (i)}] $X$ is countable, so $B_1(X,Y)=Y^X$ is separable and metrizable;
\item[{\rm (ii)}] $B_1(X,Y)$ is Lindel\"{o}f;
\item[{\rm (iii)}] $B_1(X,Y)$ is normal.
\end{enumerate}
\end{theorem}

\begin{proof}
(i)$\Rightarrow$(ii) If $X$ is countable, then, by Proposition \ref{p:B1-countable}, $B_1(X,Y)=Y^X$. Thus $B_1(X,Y)$ is separable and metrizable and hence Lindel\"{o}f. The implication (ii)$\Rightarrow$(iii) is clear.

(iii)$\Rightarrow$(i) Assume that $B_1(X,Y)$ is normal. Then, by Proposition \ref{p:B1-normal}, $X$ is $k$-scattered. By Theorem 9.3 of \cite{BG-Baire}, the space $X$ is functionally countable. Now Proposition \ref{p:B1=B2} implies that $B_1(X,Y)=B_2(X,Y)$.  Therefore the space $B_2(X,Y)$ is also normal.
By Proposition \ref{p:Y-good=>good-I}, the space $X$ is $Y$-Tychonoff. Finally, applying Theorem \ref{t:Baire-class-normal} and Lemma \ref{l:Baire-1}, we obtain that $X$ is countable. \qed
\end{proof}

To prove the next theorem we shall use the following result which is proved in \cite{K15,K16,KM}.

\begin{proposition} \label{p:Baire-classes-relation}
If $Y$ is a separable metrizable space and $X$ is $Y$-dimensional, then $B_1(X,Y)=\A_1^0(X,Y)$.
\end{proposition}

Recall that a metrizable space $X$ is called an {\em absolute retract} if $X$ is a retract of any metrizable space containing $X$ as a closed subspace.

\begin{theorem} \label{t:B1-retract-normal}
Let $Y$ be a path-connected non-compact Polish absolute retract, and let $X$ be a normal $K$-analytic space  of countable pseudocharacter. Then $B_1(X,Y)$ is a normal space if and only if $X$ is countable. In this case $B_1(X,Y)=Y^X$ is a separable metrizable space.
\end{theorem}

\begin{proof}
By Theorem 16.1(d) of Section II in \cite{Hu}, the space $X$ is $Y$-dimensional. Therefore, by Proposition \ref{p:Baire-classes-relation}, $B_1(X,Y)=\A_1^0(X,Y)$. Now Theorem \ref{t:B1-normal} applies. \qed
\end{proof}

Theorem \ref{t:B1-retract-normal} immediately implies

\begin{corollary} \label{c:B1-Polish-normal}
Let $X$ be a Lindel\"of \v{C}ech-complete space of countable pseudocharacter (for example, $X$ is a Polish space). Then $B_1(X)$ is normal if and only if $X$ is countable. In this case $B_1(X)=\IR^X$ is a separable metric space.
\end{corollary}


\section{Acknowledgement}

The second author expresses sincere thanks to Evgeny Pytkeev   for the opportunity to get acquainted with  the absolutely inaccessible PhD thesis of Pestryakov \cite{Pestryakov} and for  his kind attention to this work.




\end{document}